\documentclass{amsart}
\usepackage{graphicx}
\usepackage{amscd}
\usepackage{amsmath}
\usepackage{amsfonts}
\usepackage{amssymb}
\newtheorem{theorem}{Theorem}
\theoremstyle{plain}

\newtheorem{corollary}[theorem]{Corollary}

\newtheorem{lemma}[theorem]{Lemma}

\newtheorem{proposition}[theorem]{Proposition}

\numberwithin{theorem}{section}
\numberwithin{equation}{section}

\begin{document}
\title[Topological radicals, I.]{Topological radicals, I. Basic properties, tensor products and joint quasinilpotence}
\author{Victor S. Shulman}
\address{Vologda Polytechnical Institute, Lenin str. 15, Vologda 160000, Russia }
\email{shulman.victor80@yahoo.com}
\author{Yuri V. Turovskii}
\address{Institute of Mathematics and Mechanics, F. Agaev str. 9, Baku AZ1141,
Azerbaijan }
\email{yuri.turovskii@gmail.com}
\thanks{This paper is in final form and no version of it will be submitted for
publication elsewhere. }
\dedicatory{Dedicated to Professor Wieslaw \.{Z}elazko on the occasion of his 70-th birthday}

\begin{abstract}
The paper starts a series of publications devoted to the theory of topological
radicals (TRs) of normed algebras. It contains results on general properties
of TRs and their domains, considers TRs close to the Jacobson radical,
presents a theory of tensor TRs (they are connected with the problem of
calculating the radical of a projective tensor product), introduces and
studies TRs related to the notion of joint qusinilpotence.
\end{abstract}\maketitle

\section{\label{S1} Introduction}

Studying some property, say $\mathcal{P}$, of associative algebras it is
sometimes possible to single out the smallest ideals that accumulate
$\mathcal{P}$ (or the largest ideals of algebras that have $\mathcal{P}$).
This means that the quotients by these ideals are free of $\mathcal{P}$ in the
sense that they do not have nonzero ideals that possess $\mathcal{P}$. Thus
one obtains two, with respect to $\mathcal{P}$, classes of algebras (usually
called $\mathcal{P}$\textit{-radical} and $\mathcal{P}$\textit{-semisimple})
which can be investigated separately and then joined by means of the extension
theory. The first important examples were related to the nilpotence and some
close properties, namely lower and upper nil radicals, quasiregular radical,
etc. The intensive development of this approach brought to a rich and fruitful
branch of the modern algebra, the general theory of radicals \cite{D65, S81},
as maps associating to an algebra its ideal and satisfying some special axioms.

A topological analog of this theory, namely the theory of topological radicals
of normed algebras, was initiated by P. G. Dixon \cite{D97}. This work
contained a well thought-out axiomatics, topological versions of some basic
constructions and many interesting examples. What is especially important,
\cite{D97} proposed a radical theory approach to one of difficult problems of
the Banach algebra theory, the problem of the existence of topologically
irreducible representations of Jacobson radical Banach algebras. The solution
of this problem, obtained by C. Read \cite{R97}, was stimulated by Dixon's approach.

Here we start a series of works on topological radicals of normed (in
particular Banach) algebras. They are related to some known problems of Banach
algebra theory and operator theory: the existence of non-trivial ideals,
radicality of tensor products, joint quasinilpotence of topologically nil
algebras, invariant subspaces, scarcity of spectra, spectral theory of
multiplication operators and so on. Most of them will be completely settled
only in the presence of the (weakest of possible) compactness type conditions,
but many partial results will be obtained without such assumptions. We are
aimed also in the intrinsic development of the theory, not related directly to
the outer problems.

The present paper consists of three sections. The first one, Section \ref{S2},
considers the basic properties of topological radicals (TRs) and their
domains, the \textit{ground} and \textit{universal classes} of normed
algebras. We establish some useful general properties of TRs, classify TRs
with respect to some more special conditions, present, compare and investigate
examples of TRs related to the Jacobson radical (which itself is not a TR if
considered on the class of all normed algebras). It is worth mentioning that
the developed theory of TRs finds its first applications in the difficult
problem of distinguishing some universal classes (see Subsection \ref{S2.11}).
A part of the section is devoted to the study of extensions. They are
understood in two different senses: the stability of a ground class with
respect to the forming of extensions (does an algebra belongs to the ground
class if this class contains its ideal and the quotient?) and the possibility
to extend a TR to a larger ground class. The first direction is of technical
use (but some results seem to be valuable themself, for example the extension
stability of the class of all $Q$-algebras). The second one is of central
importance and is related to many further topics of our project (starting with
tensor products). The situation can be described as follows: to deal with a
radical on a ground class it is very useful to know if it can be extended
(with preservation of some special properties) to more wide (presumably
universal) classes. This (apart of the intrinsic beauty of the subject)
justifies our interest in non-complete algebras: the class of Banach algebras
is not universal.

Section \ref{S3} considers the behavior of a TR with respect to the forming of
projective tensor products of Banach algebras. It was stimulated by the
unsolved problem of radicality of a tensor product of a (Jacobson) radical
Banach algebra and an arbitrary one. We define the general class of
\textit{tensor} TRs, and construct the tensor radical $R^{t}$ related to a
given TR $R$ (coinciding with $R$ iff $R$ is tensor). More precisely we find
some conditions under which $R^{t}$ is a TR and satisfies some additional
properties. Then we restrict our attention to the case that
$R=\operatorname{Rad}$, the Jacobson radical on the class of all Banach
algebras, describe the properties of $\operatorname{Rad}^{t}$ and relate it
with a kind of joint quasinilpotence (with respect to an $l^{1}$-version of
the joint spectral radius).

Recall that a bounded subset $M$ of a normed algebra is called \textit{jointly
quasinilpotent} if $\left\|  M^{k}\right\|  ^{1/k}\rightarrow0$ as
$k\rightarrow\infty$, where $M^{k}$ is the set of all products of $k$ elements
from $M$ and the norm of a set is defined as supremum of norms of its
elements. It is an open problem if in a radical Banach algebra each finite (or
each precompact) set is jointly quasinilpotent. In Section \ref{S4} we
construct and study topological radicals related to the properties of joint
quasinilpotence and investigate their connections to $\operatorname{Rad}^{t}$.
In subsequent publications we will return to these notions and problems. The
authors would like to express their gratitude to Maria Fragoulopoulou for a
consultation on $Q$-algebras.

\subsection{Preliminaries}

In what follows all linear spaces and algebras are complex. For a linear space
$X$, $L(X)$ denotes the algebra of linear operators on $X$. If $X$ is normed,
$B(X)$ denotes the subalgebra of $L(X)$ consisting of all bounded operators on
$X$. The completion of a normed space $X$ is denoted by $\overline{X}$. If $X$
is a subspace of a normed space $Y$ then $\overline{X}$ is identified as a
rule with the closure of $X$ in $\overline{Y}$. So the closure of $X$ in $Y$
can be written as $\overline{X}\cap Y$.

The `\textit{unitization}' $A^{1}$ of an algebra $A$ is defined as $A$ itself
if $A$ is unital and as $A\oplus\mathbb{C}$, with standard operations,
otherwise. The term `ideal' means a two-sided ideal; note that all ideals of
$A$ are simultaneously ideals of $A^{1}$. If $I$ is an ideal of $A$ then
$q_{I}$ denotes the canonical epimorphism of $A$ onto the quotient $A/I$.
Sometimes instead of $q_{I}(a)$ (resp. $q_{I}(M)$) we write $a/I$ (resp.
$M/I$) for every $a\in A$ (resp. $M\subset A$).

Let $\operatorname{irr}A$ denote the set of all strictly irreducible
representations of $A$. If $A$ is normed, let $\operatorname{irr}_{b}A$ (resp.
$\operatorname{irr}_{n}A$) denote the set of all continuous representations in
$\operatorname{irr}A$ by bounded operators on Banach (resp. normed) spaces.
Two representations $\pi$ and $\tau$ in $\operatorname{irr}A$ on $X_{\pi}$ and
$X_{\tau}$, respectively, are called \textit{equivalent } if there exists an
isomorphism $T:X_{\pi}\rightarrow X_{\tau}$ such that $\tau(a)T=T\pi(a)$ for
every $a\in A$. It is known that any strictly irreducible representation of an
algebra $A$ is equivalent to the left regular representation $\pi^{M}$ on the
quotient space $A/M$, where $M$ is a maximal modular left ideal. Let
$\operatorname*{Prim}A$ denote the set of primitive ideals of $A$ (i.e. the
kernels of representations in $\operatorname{irr}A$). By definition, the
\textit{Jacobson radical} $\operatorname{rad}(A)$ is the intersection of all
ideals in $\operatorname*{Prim}A^{1}$. There are many equivalent algebraic
characterizations of $\operatorname{rad}(A)$; in particular it is equal to the
intersection of all modular maximal left ideals. Also, $\operatorname{rad}(A)$
is the largest quasi-regular ideal (i.e. all its elements are
quasi-invertible; $a\in A$ is \textit{quasi-invertible} if $1-a$ is invertible
in $A^{1}$).

The \textit{spectrum} ${\sigma}_{A}(a)$ of an element $a\in A$ is the set of
all $\lambda\in\mathbb{C}$ for which $a-{\lambda}$ is not invertible in
$A^{1}$ (this definition agrees really with one via quasi-inverses in
\cite[Section 2.1]{P94}). The \textit{algebraic spectral radius} ${\rho}%
_{A}(a)$ is defined as $\sup\{\left|  \lambda\right|  :\lambda\in\sigma
_{A}(a)\}$. If $A$ is a normed algebra then ${\sigma}_{A}(a)$ is always
nonempty (\cite[Theorem 2.2.2]{P94}). In this case one defines also the
\textit{topological spectral radius} $\rho(a)=\lim\Vert a^{n}\Vert^{1/n}%
=\inf\Vert a^{n}\Vert^{1/n}$. In general $\rho(a)\leqslant\rho_{A}(a)$ and
$\rho(a)=\rho_{\overline{A}}(a)$ by Gelfand's spectral radius formula, where
$\overline{A}$ is the completion of $A$. If $\rho(a)=0$, we say that $a$ is
\textit{quasinilpotent}. An (one-sided) ideal is \textit{topologically nil}
\cite{P94} if it consists of quasinilpotent elements.

\section{\label{S2}Topological radicals, classes of normed algebras, and the
Jacobson radical}

In this section we investigate the properties of TRs defined on different
classes of normed algebras. First of all we consider normed $Q$-algebras.

\subsection{$Q$-algebras}

A normed algebra $A$ is called a $Q$\textit{-algebra} if the set of all
invertible elements of $A^{1}$ is open. Several equivalent characterizations
of this important property can be found in \cite{P94}. In particular, the
following are equivalent \cite[Proposition 2.2.7]{P94} for normed algebras.

\begin{itemize}
\item [(Q$_{1}$)]$A$ is a $Q$-algebra.

\item[(Q$_{2}$)] $\rho_{A}(a)=\rho(a)$ for any $a\in A$.

\item[(Q$_{3}$)] $\sum_{n>0}a^{n}$ converges for any $a\in A$ with $\Vert
a\Vert<1$.
\end{itemize}

It follows that in any $Q$-algebra the norm is \textit{spectral}, i.e.,
$\rho_{A}(a)\leqslant\Vert a\Vert$; conversely, an algebra with spectral norm
is a $Q$-algebra. If $A$ is a subalgebra of an algebra $B$, $A$ is called a
\textit{spectral subalgebra} of $B$ \cite[Definition 2.5.1]{P94} if
$\sigma_{A}(a)\backslash\{0\}=\sigma_{B}(a)\backslash\{0\}$ for every $a\in A$.

Our aim here is to prove that $Q$-algebras can be conveniently characterized
in terms of their strictly irreducible representations.

\begin{theorem}
\label{QA} For a normed algebra $A$, $(\operatorname{Q}_{1})$ is equivalent to
the following conditions.

\begin{itemize}
\item [(Q$_{4}$)]Each maximal modular left ideal of $A$ is closed.

\item[(Q$_{5}$)] Each strictly irreducible representation of $A$ is equivalent
to a continuous representation (by bounded operators on a normed space).
\end{itemize}
\end{theorem}

\begin{proof}
The implication (Q$_{1}$)$\Rightarrow$(Q$_{4}$) is well known \cite[Theorem
2.2.8]{D00}. (Q$_{4}$)$\Rightarrow$(Q$_{5}$) follows from the fact that any
strictly irreducible representation of $A$ is equivalent to some
representation $\pi^{M}$, where $M$ is a maximal modular left ideal of $A$. If
$M$ is closed then the space $A/M$ obtains the quotient norm with respect to
which $\pi^{M}$ is clearly continuous.

Now let (Q$_{5}$) hold; we will show (Q$_{5}$)$\Rightarrow$(Q$_{2}$). It
suffices to prove that $A$ is a spectral subalgebra of $\overline{A}$. Let
$a\in A$ and, for some $\lambda\neq0$, $a-{\lambda}$ be non-invertible in
$A^{1}$. We should prove that it is not invertible in $B^{1}$, where
$B=\overline{A}$.

If $a-{\lambda}$ is not left invertible in $A^{1}$ then there is a maximal
left ideal $M$ of $A^{1}$ containing $a-{\lambda}$. Let $x=1+M\in A^{1}/M$.
Then
\[
\pi^{M}(a)x={\lambda x.}%
\]
Since $\pi^{M}(A)\neq0$, the restriction of $\pi^{M}$ to $A$ is strictly
irreducible. We proved that there is a strictly irreducible representation
$\pi$ of $A$ such that $\lambda$ is an eigenvalue of $\pi(a)$. By the
assumption, we can suppose that $\pi$ is a continuous representation by
bounded operators on a normed space $X$. Let $Y=\overline{X}$, the completion
of $X$, and let $\overline{{\pi}}$ the representation of $A$ on $Y$ such that
$\overline{{\pi}}(a)$ is the extension of $\pi(a)$ by continuity, for any
$a\in A$. Again $\overline{{\pi}}$ extends by continuity to a representation
$\tau$ of $B$ which in its turn extends to a representation ${\tau}^{\prime}$
of $B^{1}$ on the same space. It is easy to see that $\lambda$ is an
eigenvalue of ${\tau}^{\prime}(a)$. But this means that the element $a$ cannot
be left invertible in $B^{1}$.

So we may suppose that $a-\lambda$ is left invertible in $A^{1}$:
\[
c(a-{\lambda})=1
\]
for some $c\in A^{1}$. If $a-\lambda$ is invertible in $B^{1}$ then
\[
(a-{\lambda})b=1
\]
for some $b\in B^{1}$, whence
\[
c=c(a-{\lambda})b=b
\]
and accordingly $b\in A^{1}$, a contradiction. This shows that $A$ is a
spectral subalgebra of $\overline{A}$.
\end{proof}

We continue the list of conditions equivalent to (Q$_{1}$).

\begin{itemize}
\item [(Q$_{6}$)]$A$ is a spectral subalgebra of $\overline{A}$.

\item[(Q$_{7}$)] Every strictly irreducible representation of $A$ extends to a
representation of $\overline{A}$.
\end{itemize}

The equivalence of the conditions (Q$_{6}$) and (Q$_{7}$) is a special case of
the Schweitzer's Theorem \cite[Theorem 4.2.10]{P94}, the equivalence of
(Q$_{6}$) and (Q$_{1}$) is well known (for instance see \cite[Lemma
20.9]{KS97}).

Looking at the condition (Q$_{5}$) of Theorem \ref{QA} as a possible
definition of a $Q$-algebra it is natural to consider normed algebras whose
strictly irreducible representations are equivalent to continuous
representations (by bounded operators) on Banach spaces. We call them $Q_{b}%
$-\textit{algebras}. Being related to a special radical introduced in
\cite{D97}, these algebras play an important role in what follows.

\subsection{Classes of normed algebras}

Let $(NA)$ denote the class of all normed algebras. The following definitions
will be useful.

A class $\mathcal{A}\subset(NA)$ is called \textit{image closed }(resp.
\textit{preimage closed}) if conditions $B=f(A)$ for continuous isomorphism
$f$ and $A\in\mathcal{A}$ (resp. $B\in\mathcal{A}$) imply $B\in\mathcal{A}$
(resp. $A\in\mathcal{A}$). Recall that a continuous isomorphism is a
\textit{topological isomorphism }if its inverse is also continuous.

A class $\mathcal{A}\subset(NA)$ is called \textit{ideal stable }(resp.
\textit{closed ideal stable}) if it contains all normed algebras topologically
isomorphic to ideals (resp. closed ideals) of every $A\in\mathcal{A}$. A class
$\mathcal{A}\subset(NA)$ is called \textit{quotient stable }if it contains all
normed algebras topologically isomorphic to quotients of every $A\in
\mathcal{A}$ by closed ideals.

A class $\mathcal{A}\subset(NA)$ is called a \textit{ground class} if it is
closed ideal and quotient stable, and a \textit{universal class }if it is
ideal and quotient stable. Note that ground and universal classes of normed
algebras are natural domains of topological radicals.

In many cases the properties of synthetic kind are important. For
$\mathcal{A}\subset\mathcal{B}\subset(NA)$, $\mathcal{A}$ is called
\textit{extension stable }in $\mathcal{B}$ if conditions $A\in\mathcal{B}$,
$I$ is a closed ideal of $A$ and $A/I,I\in\mathcal{A}$ imply $A\in\mathcal{A}%
$. We say that $\mathcal{A}$ is \textit{linear} in $\mathcal{B}$ if every
$A\in\mathcal{B}$ with a dense sum of its ideals $I_{\alpha}$ that belong to
$\mathcal{A}$ belongs to $\mathcal{A}$ itself. In both cases we do not mention
$\mathcal{B}$ if $\mathcal{B}=(NA)$.

Let us now consider some examples of ground classes.

\begin{itemize}
\item  The class $(BA)$ of all Banach algebras. It is an easy exercise that
the class is extension stable, not linear and not ideal stable.

\item  The class $(QA)$ (resp. $(Q_{b}A)$) of all normed $Q$-algebras (resp.
$Q_{b}$-algebras). Classes $(QA)$ and $(Q_{b}A)$ are universal (see Theorem
\ref{p20}).

\item  The smallest universal class $(BA)^{u}$ containing all Banach algebras.

\item  The class $(CNA)$ of all commutative normed algebras. The class is
universal, not extension stable and not linear.

\item  The class $(C^{\ast}EA)$ of all $C^{\ast}$-equivalent algebras. The
class is not universal. We will investigate it elsewhere.
\end{itemize}

Restricting ``dimension'' of algebras in a ground class we get another ground
class. In this way one obtains the classes $(SBA)$ and $(SNA)$ of all
separable Banach and, respectively, normed algebras and the class $(FNA)$ of
all finite-dimensional normed algebras. There are many possibilities to form
new ground classes starting from the given ones.

\begin{proposition}
\label{p4} Unions and intersections of arbitrary families of ground (resp.
universal) classes are ground (resp. universal) classes.
\end{proposition}

\begin{proof}
Clear.
\end{proof}

In particular the intersection of any ground class with $(CNA)$ is often important.

It is easy to see that%
\[
(C^{\ast}EA)\subset(BA)\subset(BA)^{u}\subset(QA_{b})\subset(QA)\subset(NA).
\]

\subsection{Classes $(QA)$ and $(Q_{b}A)$}

Now let us return to $(Q_{b}A)$ and $(QA)$. Algebraic parts of the following
two lemmas belong to the folklore and we are mainly interested in the
topological aspect. Let $A$ be an algebra and $I$ its ideal. Put
$\operatorname{irr}_{I}A=\{\pi\in\operatorname{irr}A:\pi I=0\}$ and
$\operatorname{irr}^{I}A=\{\pi\in\operatorname{irr}A:\pi I\neq0\}$. Let
$q_{I}$ be the standard epimorphism $A\rightarrow A/I$, and let $q_{I}%
^{-1}(a)$ be the preimage of $a\in A/I$.

\begin{lemma}
\label{p18}Let $A\in(NA)$ and $I$ its closed ideal. If $A$ belongs to
$(Q_{b}A)$ (resp. $(QA)$) then so does $A/I$. If $A/I\in(Q_{b}A)$ (resp.
$(QA)$) then every representation from $\operatorname{irr}_{I}A$ is equivalent
to a continuous representation in $\operatorname{irr}_{I}A$ by bounded
operators on a Banach (resp. normed) space.
\end{lemma}

\begin{proof}
Consider firstly $A$ and $I$ in the pure algebraic context (i.e. without any
topology). For arbitrary $\pi\in\operatorname{irr}_{I}A$ and $\tau
\in\operatorname{irr}A/I$, put $\widehat{\pi}=\pi q_{I}$ and $\widetilde{\tau
}=\tau q_{I}^{-1}$. Clearly $\widehat{\pi}\in\operatorname{irr}A/I$ and
$\widetilde{\tau}\in\operatorname{irr}_{I}A$. If $\tau^{\prime}$ is equivalent
to $\tau$ and $\widetilde{\tau}$ is equivalent to $\pi$ then $\tau^{\prime}$
is equivalent to $\widehat{\pi}$. Indeed, identifying the representation
spaces pairly, one can identify the actions of $\tau^{\prime}$ and
$\widehat{\pi}$. Similarly, if $\pi^{\prime}$ is equivalent to $\pi$ and
$\widehat{\pi}$ is equivalent to $\tau$ then $\pi^{\prime}$ is equivalent to
$\widetilde{\tau}$.

If $A$ is normed, $I$ is closed and $\pi$ (resp. $\tau$) is a bounded
representation by bounded operators on a normed (Banach) space, we have that
$\widehat{\pi}$ (resp. $\widetilde{\tau}$) is a representation by bounded
operators and is bounded. Indeed, $\left\|  \widehat{\pi}\right\|
\leqslant\left\|  \pi\right\|  \left\|  q_{I}\right\|  $ and, for every $a\in
A$,
\[
\left\|  \widetilde{\tau}a\right\|  \leqslant\left\|  \tau\right\|  \left\|
q_{I}^{-1}(a)\right\|  \leqslant\left\|  \tau\right\|  \left\|  a\right\|  ,
\]
whence $\left\|  \widetilde{\tau}\right\|  \leqslant\left\|  \tau\right\|  $.
\end{proof}

\begin{lemma}
\label{p19} Let $A\in(NA)$ and $I$ its ideal. If $A$ belongs to $(Q_{b}A)$
(resp. $(QA)$) then so does $I$. If $I\in(Q_{b}A)$ (resp. $(QA)$) then every
representation from $\operatorname{irr}^{I}A$ is equivalent to a continuous
representation in $\operatorname{irr}^{I}A$ by bounded operators on a Banach
(resp. normed) space.
\end{lemma}

\begin{proof}
Consider firstly $A$ and $I$ in the pure algebraic context. For arbitrary
$\pi\in\operatorname{irr}^{I}A$ and $\tau\in\operatorname{irr}I$, let
$\pi|_{I}$ be the restriction of $\pi$ to $I$ and, for arbitrary nonzero $x\in
X_{\tau}$, $y=\tau(b)x$ with $b\in I$, let $\tau_{x}$ be defined as $\tau
_{x}(a)y=\tau(ab)x$ for every $a\in A$. The definition of $\tau_{x}$ is
correct: if $y=0$ then $\tau(I)\tau(ab)x=\tau(Ia)y=0$, whence $\tau(ab)x=0$.
Clearly $\pi|_{I}\in\operatorname{irr}I$ and $\tau_{x}\in\operatorname{irr}%
^{I}A$, and, for fixed $\tau$, all $\tau_{x}$ are pairly equivalent.

If $\pi^{\prime}$ is equivalent to $\pi$ and $\pi|_{I}$ is equivalent to
$\tau$ then $\pi^{\prime}$ is equivalent to $\tau_{x}$, for every $x$. Indeed,
identifying $\pi^{\prime}$ with $\pi$ and $\tau$ with $\pi|_{I}$, one can
assume that the representations act on the same space, say $X$ and, for every
$y=\tau(b)x=\pi|_{I}(b)x=\pi(b)x\in X$ with $b\in I$ and for every $a\in A$,
we have
\[
\tau_{x}(a)y=\tau(ab)x=\pi|_{I}(ab)x=\pi(ab)x=\pi(a)y.
\]
Similarly, if $\tau^{\prime}$ is equivalent to $\tau$ and $\tau_{x}$ is
equivalent to $\pi$ then $\tau^{\prime}$ is equivalent to $\pi|_{I}$. Indeed,
identifying $\tau^{\prime}$ with $\tau$ and $\tau_{x}$ with $\pi$, one can
assume that the representations act on the same space, say $X$ and, for every
$y=\tau(b)x\in X$ with $b\in I$ and for every $a\in I$, we have
\[
\pi|_{I}(a)y=\pi(a)y=\tau_{x}(a)x=\tau(ab)x=\tau(a)y.
\]
If $A$ is normed and $\pi$ (resp. $\tau$) is a bounded representation by
bounded operators on a normed (Banach) space $X$, we have that $\pi|_{I}$
(resp. $\tau_{x}$) is a representation by bounded operators and is bounded. It
suffices to check boundedness for $\tau_{x}$ if $\tau$ is bounded. For every
$y\in X$, put
\[
\left\|  y\right\|  _{(I_{,}x)}=\inf\{\left\|  b\right\|  :b\in I,\quad
\tau(b)x=y\}.
\]
Then $y\mapsto\left\|  y\right\|  _{(I_{,}x)}$ determines a new norm on $X$.
Indeed, if $\left\|  y\right\|  _{(I_{,}x)}=0$ then $\left\|  y\right\|
=\left\|  \tau(b)x\right\|  \leqslant\left\|  \tau\right\|  \left\|
b\right\|  \left\|  x\right\|  $ implies $\left\|  y\right\|  =0$, i.e. $y=0$;
the other properties of norm are obvious for $y\mapsto\left\|  y\right\|
_{(I_{,}x)}$. We have that, for every $a\in A$,
\[
\left\|  \tau_{x}(a)y\right\|  _{(I_{,}x)}\leqslant\left\|  \tau(ab)x\right\|
_{(I,x)}\leqslant\left\|  ab\right\|  \leqslant\left\|  a\right\|  \left\|
b\right\|  ,
\]
whence
\[
\left\|  \tau_{x}(a)y\right\|  _{(I_{,}x)}\leqslant\left\|  a\right\|
\left\|  y\right\|  _{(I_{,}x)},\quad\left\|  \tau_{x}(a)\right\|  _{(I_{,}%
x)}\leqslant\left\|  a\right\|  ,\quad\left\|  \tau_{x}\right\|  _{(I_{,}%
x)}\leqslant1.
\]
If $X$ was a Banach space, take $\overline{I}$ instead of $I$ and
$\overline{\tau}$, the continuous extension of $\tau$ to $\overline{I}$ by
bounded operators on $X$, instead of $\tau$. Then the above estimates hold,
and it remains to show that $(X,\left\|  \cdot\right\|  _{(\overline{I}_{,}%
x)})$ is complete. Indeed, $(X,\left\|  \cdot\right\|  _{(\overline{I}_{,}%
x)})$ is isometrically isomorphic to the quotient $\overline{I}/M$, where
$M=\{b\in\overline{I}:\overline{\tau}(b)x=0\}$, and hence is complete.
\end{proof}

Now we are in a position to describe the basic properties of $(Q_{b}A)$ and
$(QA)$.

\begin{theorem}
\label{p20}$(Q_{b}A)$ and $(QA)$ are extension stable, preimage closed, linear
universal classes.
\end{theorem}

\begin{proof}
It follows from Lemmas \ref{p18} and \ref{p19} that $(Q_{b}A)$ and $(QA)$ are
extension stable and universal.

Let $A\in(NA)$, $B\in(QA)$ and $f(A)=B$ for some continuous isomorphism $f$.
Then
\[
\rho(a)\geqslant\rho(f(a))=\rho_{B}(f(a))=\rho_{A}(a)
\]
and therefore $\rho_{A}(a)=\rho(a)$ for every $a\in A$. Thus $A\in(QA)$.

Suppose now that $B\in(Q_{b}A)$. Take an arbitrary $\pi\in\operatorname{irr}%
A$. Then $\pi f^{-1}\in\operatorname{irr}B$ is equivalent to some $\tau
\in\operatorname{irr}_{b}B$. Hence $\tau f\in\operatorname{irr}_{b}A$. Indeed,
$\tau f(a)x=\tau(f(a))x$ for every $a\in A$ and $x\in X_{\tau}$, whence $\tau
f(a)\in B(X_{\tau})$ and also $\left\|  \tau f\right\|  \leqslant\left\|
\tau\right\|  \left\|  f\right\|  $. It is easy to see that $\pi$ is
equivalent to $\tau f$. So $A\in(Q_{b}A)$. This shows that $(Q_{b}A)$ and
$(QA)$ are preimage closed.

Let $A\in(NA)$ be the closure of sum of its ideals $I_{\alpha}\in(Q_{b}A)$
(resp. $(QA)$), and let $\pi\in\operatorname{irr}A$ be arbitrary. Then there
exists an index $\beta$ such that $\pi I_{\beta}\neq0.$ By Lemma \ref{p19},
$\pi$ is equivalent to a continuous representation in $\operatorname{irr}%
^{I_{\beta}}A$ by bounded operators on a Banach (resp. normed) space.
Therefore $A\in(Q_{b}A)$ (resp. $(QA)$). We proved that $(Q_{b}A)$ and $(QA)$
are linear.
\end{proof}

The facts that $(QA)$ is universal and preimage closed are known \cite{P94}.
Also, I. Kaplansky \cite[Lemma 3]{K48} proved that $(QA)$ is extension stable.
An ideal of a normed algebra is called a $Q$\textit{-ideal} (resp. $Q_{b}%
$\textit{-ideal}) if it is a $Q$-algebra (resp. $Q_{b}$-algebra).

\begin{corollary}
Every normed algebra has the largest $Q$-ideal (resp. $Q_{b}$-ideal); the
latter is closed.
\end{corollary}

\begin{proof}
Let $A\in(NA)$ and let $J$ be the closure of sum of all $Q$-ideals (resp.
$Q_{b}$-ideals) $I_{\alpha}$ of $A$. Then every $I_{\alpha}$ is also an ideal
of $J$. It follows from Theorem \ref{p20} that $J\in(QA)$ (resp. $(Q_{b}A)$).
\end{proof}

Extension stability and linearity of $(QA)$ solve in the context of normed
algebras some Palmer's questions on spectral algebras \cite[Page 234]{P94}.

\subsection{\label{sub0}Topological radicals and their elementary properties}

Recall that an epimorphism $f:A\rightarrow B$ of normed algebras is
\textit{open} (i.e., the images of open sets are open) iff there is a constant
$C>0$ such that for any $b\in B$ there is $a\in A$ with $f(a)=b$ and $\Vert
a\Vert\leqslant C\Vert b\Vert$. An example of an open continuous epimorphism
is a quotient map $q_{I}:A\rightarrow A/I$, where $I$ is a closed ideal of
$A$. It is not difficult to check that any open continuous epimorphism $f$ is
a composition of a topological isomorphism and the quotient map $q_{\ker f}$ .

Let $R$ be a map associating with any algebra $A\in\mathcal{A}$ its closed
ideal $R(A)$. It is called \textit{a topological radical} on a ground class
$\mathcal{A}$ ($\mathcal{A}$\textit{-radical}, in short) if the following
conditions are satisfied.

\begin{itemize}
\item [(1$^{\circ}$)]$R(R(A))=R(A)$, for any $A\in\mathcal{A}$.

\item[(2$^{\circ}$)] $R(A/R(A))=0$, for any $A\in\mathcal{A}$.

\item[(3$^{\circ}$)] $f(R(A))=R(B)$, for any topological isomorphism
$f:A\rightarrow B$ with $A,B\in\mathcal{A}$.

\item[(4$^{\circ}$)] $q_{I}(R(A))\subset R(A/I)$, for any closed ideal $I$ of
$A\in\mathcal{A}$.

\item[(5$^{\circ}$)] If an ideal $I$ of $A\in\mathcal{A}$ belongs to
$\mathcal{A}$ then

\begin{itemize}
\item [(5$_{1}^{\circ}$)]$R(I)$ is an ideal of $A$.

\item[(5$_{2}^{\circ}$)] $R(I)\subset I\cap R(A)$.
\end{itemize}
\end{itemize}

If the class $\mathcal{A}$ is obvious we simply say that $R$ is a
\textit{topological radical }(TR). A TR $R$ is called \textit{a hereditary
topological radical} (HTR) if it satisfies the more strong than (5$^{\circ}$) condition:

\begin{itemize}
\item [(6$^{\circ}$)]$R(I)=I\cap R(A)$, for any ideal $I\in\mathcal{A}$ of
$A\in\mathcal{A}$.
\end{itemize}

If $R$ satisfies the conditions (1$^{\circ}$), (3$^{\circ}$), (4$^{\circ}$)
and (5$^{\circ}$) (respectively (2$^{\circ}$), (3$^{\circ}$), (4$^{\circ}$),
(5$^{\circ}$)) then it is called an \textit{under topological radical} (UTR)
(respectively \textit{over topological radical} (OTR)). If (6$^{\circ}$) holds
then a UTR is also called hereditary (note that a hereditary OTR is really an HTR).

The terms UTR and OTR were suggested by P. G. Dixon in virtue of \cite[Theorem
6.11]{D97}; \cite[Theorems 6.6 and 6.10]{D97} also clarified the reason for
this terminology (note that the prefixes `lower' and `upper' are overloaded,
while the prefixes `sub' and `super' mean commonly something another).

Let $R$ be an $\mathcal{A}$-radical. An algebra $A\in\mathcal{A}$ is called
$R$\textit{-semisimple} if $R(A)=0$; $A$ is $R$\textit{-radical} if $R(A)=A$.
Clearly the ideals of $R$-semisimple algebras are $R$-semisimple and quotients
of $R$-radical algebras by closed ideals are $R$-radical. If $R$ is an HTR
then the ideals of an $R$-radical algebra are $R$-radical. The converse is
always true: any $R$-radical ideal of $A$ is contained in $R(A)$. This can be
formulated in the following way.

\begin{lemma}
\label{2.1} Let $R$ be a TR on a ground class $\mathcal{A}$. Then, for each $A
\in\mathcal{A}$,

\begin{itemize}
\item [(i)]$R(A)$ is the largest $R$-radical ideal of $A$.

\item[(ii)] $R(A)$ is the smallest closed ideal of $A$ with $R$-semisimple quotient.
\end{itemize}
\end{lemma}

\begin{proof}
(i) If $R(I)=I$ then the inclusion $I\subset R(A)$ follows from (5$^{\circ}$).

(ii) If $R(A/I)=0$ then, by (4$^{\circ}$), $q_{I}(R(A))=0$ and $R(A)\subset I$.
\end{proof}

\begin{corollary}
\label{2.2} Let $R$ be a TR on a ground class $\mathcal{A}$.

\begin{itemize}
\item [(i)]The closure of an $R$-radical ideal is $R$-radical.

\item[(ii)] If $I$ is a closed $R$-radical ideal of $A$ then $q_{I}%
(R(A))=R(A/I)$.
\end{itemize}
\end{corollary}

\begin{proof}
Let $I$ be an $R$-radical ideal of an algebra $A\in\mathcal{A}$, and let $J$
be the closure of $I$ in $A$. Then $J$ is a closed ideal of $A$, in particular
$J\in\mathcal{A}$, and $I$ is an $R$-radical ideal of $J$. Since $R(J)$ is
closed in $J$ and contains an ideal dense in $J$, $R(J)=J$. This proves (i).

For (ii), suppose that $I$ is a closed $R$-radical ideal of $A$. Then
$I\subset R(A)$ and there is an open continuous epimorphism $p:A/I\rightarrow
A/R(A)$ such that $q_{R(A)}=pq_{I}$. Hence $p(R(A/I))\subset R(A/R(A))=0$,
whence $R(A/I)\subset\ker p=q_{I}(R(A))$. The converse inclusion follows from
(5$^{\circ}$).
\end{proof}

The following results on radicals will be useful.

\begin{theorem}
\label{2.3} Let $R$ be a TR on a ground class $\mathcal{A}$.

\begin{itemize}
\item [(i)]The class of all $R$-semisimple (resp. $R$-radical) algebras is
extension stable in $\mathcal{A}$.

\item[(ii)] If $A\in\mathcal{A}$ and every nonzero quotient of $A$ by a closed
ideal contains a nonzero $R$-radical ideal then $A$ is $R$-radical.
\end{itemize}
\end{theorem}

\begin{proof}
(i) Let $I$ be a closed ideal of a normed algebra $A$.

Suppose that $A/I$ and $I$ are $R$-semisimple. Since $A/I$ is $R$-semisimple,
$I$ contains $R(A)$ by Lemma \ref{2.1}. Hence $R(A)$ is an ideal of an
$R$-semisimple algebra $I$. So $R(A)$ is $R$-semisimple, $R(A)=R(R(A))=0$.

Suppose that $A/I$ and $I$ are $R$-radical. Since $I$ is $R$-radical, it is
contained in $R(A)$ by Lemma \ref{2.1}. Let $p:A/I\rightarrow A/R(A)$ be the
open continuous epimorphism such that $q_{R(A)}=pq_{I}$. Then
\[
A/R(A)=p(A/I)=p(R(A/I))\subset R(A/R(A))=0,
\]
whence $R(A)=A$.

(ii) If $A$ is not $R$-radical then $A/R(A)$ is $R$-semisimple and contains a
nonzero $R$-radical, a contradiction.
\end{proof}

\begin{theorem}
\label{p24}Let $R$ be a TR on a ground class $\mathcal{A}$.

\begin{itemize}
\item [(i)]The class of all $R$-radical algebras is linear in $\mathcal{A}$.

\item[(ii)] If $\mathcal{A}$ is universal and $I$ is a sum of $R$-radical
ideals $I_{\alpha}$ of $A\in\mathcal{A}$ then $I$ is $R$-radical.
\end{itemize}
\end{theorem}

\begin{proof}
Let $I\in\mathcal{A}$ be the closure of a sum of $R$-radical ideals
$I_{\alpha}$ in the case (i), or simply a sum of ones in the case (ii). In
both cases $I\in\mathcal{A}$. Since $J_{\alpha}$ is an ideal of $I$, then
$I_{\alpha}=R(I_{\alpha})\subset R(I)$ for every $\alpha$. Therefore $R(I)$ is
dense in $I$ (and also closed in $I$), whence $I=R(I)$.
\end{proof}

\begin{theorem}
\label{p26}Let $R$ be a TR on a ground class $\mathcal{A}$. Let $A\in
\mathcal{A}$ be arbitrary, $I_{\alpha}$ its closed ideals of $A$ and $I=\cap
I_{\alpha}$. If $A/I_{\alpha}$ is $R$-semisimple for every $\alpha$ then $A/I$
is $R$-semisimple.
\end{theorem}

\begin{proof}
Since $I\subset I_{\alpha}$, there exists an open continuous epimorphism
$p_{\alpha}:A/I\rightarrow A/I_{\alpha}$ such that $q_{I_{\alpha}}=p_{\alpha
}q_{I}$. By (3$^{\circ}$) and (4$^{\circ}$),%
\[
p_{\alpha}(R(A/I))\subset R(A/I_{\alpha})=0,
\]
whence $R(A/I)\subset\ker p_{\alpha}=I_{\alpha}/I$ and $q_{I}^{-1}%
(R(A/I))\subset q_{I}^{-1}(I_{\alpha}/I)=I_{\alpha}$ for every $\alpha$. So%
\[
q_{I}^{-1}(R(A/I))\subset\cap I_{\alpha}=I
\]
and accordingly $R(A/I)=0$.
\end{proof}

We remark that the quotient of an $R$-semisimple algebra by a closed ideal
need not be $R$-semisimple \cite[Page 135]{BD73}, for a TR $R$.

\subsection{Strong, strict and strictly hereditary radicals}

The properties (3$^{\circ}$) and (4$^{\circ}$) describe the behavior of a TR
under `morphisms': they both can be written as
\begin{equation}
f(R(A))\subset R(B) \tag{4$_f^{\circ}$}\label{g1}%
\end{equation}
whenever $B=$ $f(A)$. In our setting morphisms are the compositions of
topological isomorphisms and quotient maps, i.e. open continuous epimorphisms.
It is sometimes reasonable to choose a wider classes of morphisms, P. G. Dixon
\cite{D97} considered all continuous epimorphisms as morphisms in $(NA)$. This
increases the strength of a radical, but considerably reduces the lists of
radicals. On the other hand in some cases it is natural to consider more
special morphisms, for example $\ast$-epimorphisms on the class $(C^{\ast}A)$
of all $C^{\ast}$-algebras. This class will be investigated later on.\textit{ }

Now we give some related definitions. Let us say that a TR $R$ is a
\textit{strong radical }on a ground class $\mathcal{A}$ if

\begin{itemize}
\item [(7$^{\circ}$)]$f(R(A))\subset R(B)$ for each continuous epimorphism $f$
of algebras in $\mathcal{A}$.
\end{itemize}

We say that a TR $R$ is a \textit{ strict radical} on a ground class
$\mathcal{A}$ if

\begin{itemize}
\item [(8$^{\circ}$)]$f(R(A))=R(B)$ for every continuous isomorphism
$f:A\rightarrow B$ of algebras in $\mathcal{A}$.
\end{itemize}

It is not difficult to see that every strict radical is strong. It follows
from the Open Mapping Theorem that all topological radicals on $(BA)$ are
strict. The following result shows that for radicals on $(NA)$ this condition
means actually the algebraic nature of a radical.

\begin{theorem}
\label{indep} Let $R$ be a TR on the class $(NA)$. Then $R$ is strict iff $R$
does not depend on the choice of a norm, and iff $(8^{\circ})$ holds for any
algebraic isomorphism of normed algebras.
\end{theorem}

\begin{proof}
Suppose that $R$ is strict. Let $\left\|  \cdot\right\|  _{1}$ and $\left\|
\cdot\right\|  _{2}$ be norms on an algebra $A$; we define a norm $\left\|
\cdot\right\|  $ on $A$ setting $\left\|  a\right\|  =\max(\left\|  a\right\|
_{1},\left\|  a\right\|  _{2})$. Let $A$, $A_{1}$ and $A_{2}$ be the
corresponding normed algebras. Taking the identity map as a continuous
isomorphism from $A$ to $A_{m}$, $m=1,2$, we get
\[
R(A_{1})=R(A)=R(A_{2}).
\]
Therefore $R$ does not depend on a norm.

Now if $f:A\rightarrow B$ is an algebraic isomorphism of normed algebras then,
denoting by $A_{f}$ the algebra $A$ with norm $\left\|  a\right\|
_{f}=\left\|  f(a)\right\|  $, we have that $f$ is a topological isomorphism
from $A_{f}$ onto $B$ and
\[
R(B)=f(R(A_{f}))=f(R(A)).
\]
So (8$^{\circ}$) holds for algebraic isomorphisms.

The converse is evident.
\end{proof}

In some applications (for instance for extensions of radicals, tensor products
etc.) it is useful to consider as a `morphism' $A\rightarrow B$ a continuous
homomorphism of $A$ onto an ideal of $B$. We need to introduce the
corresponding definitions. A TR $R$ on a ground class $\mathcal{A}$ is called
\textit{ideally strong }if

\begin{itemize}
\item [(7$_{i}^{\circ}$)]$f(R(A))\subset R(B)$ for each continuous epimorphism
$f$ of an algebra $A\in\mathcal{A}$ to an ideal of an algebra $B\in
\mathcal{A}$.
\end{itemize}

Also, a TR $R$ on a ground class $\mathcal{A}$ is called \textit{strictly
hereditary }if

\begin{itemize}
\item [(8$_{i}^{\circ}$)]$f(R(A))=I\cap R(B)$ for each continuous isomorphism
$f$ of an algebra $A\in\mathcal{A}$ onto an ideal $I$ of an algebra
$B\in\mathcal{A}$.
\end{itemize}

Clearly every strictly hereditary TR is ideally strong. Also, every strictly
hereditary TR is strict and hereditary (to see it, take $A=I$ and the identity
map $f$ from $A$ onto $I$ in (8$_{i}^{\circ}$), for an ideal $I$ of $B$), and
every ideally strong TR is strong. The converse also holds if $\mathcal{A}$ is universal.

\begin{theorem}
\label{p21}Let $R$ be a TR on a universal class $\mathcal{A}$. If $R$ is
strong (resp. strict and hereditary) on $\mathcal{A}$ then $R$ is ideally
strong (resp. strictly hereditary) on $\mathcal{A}$.
\end{theorem}

\begin{proof}
It is an easy checkup.
\end{proof}

Recall that, by definition, a radical on the class of rings satisfies the
axioms above in which the words connected with topology must be omitted (see
\cite[Section 6]{D97} and \cite{D65}; see also axioms of radicals in the sense
of Amitsur and Kurosh in \cite{S81}).

\begin{theorem}
\label{p22}Let $R$ be a radical (resp. hereditary radical) on the class all
algebras. If $R(A)$ is closed for every algebra $A$ in a ground class
$\mathcal{A}$ then $R$ is an ideally strong TR (resp. strictly hereditary TR)
on $\mathcal{A}$.
\end{theorem}

\begin{proof}
Clear because $R$ satisfies axioms with algebraic `morphisms' applied to the
algebras from $\mathcal{A}$.
\end{proof}

The most popular and important example of a hereditary radical on the class of
all algebras (even rings) is the Jacobson radical $\operatorname{rad}$. It
should be stressed that $\operatorname{rad}$ is not a topological radical on
$(NA)$, because there are normed algebras $A$ with non-closed
$\operatorname{rad}(A)$ \cite[Example 10.1]{D97}. For $Q$-algebras this
obstacle vanishes: since every maximal modular left ideal of a $Q$-algebra $A$
is closed, $\operatorname{rad}(A)$ is also closed. As a consequence, we obtain
the following statement.

\begin{corollary}
\label{p2} $\operatorname{rad}$ is a strictly hereditary TR on $(QA)$.
\end{corollary}

Since a union of ground (resp. universal) classes is again ground (resp.
universal), there exists a maximal ground (resp. universal) class on which
$\operatorname{rad}$ is a TR. Let us denote it by $(\operatorname{rad})^{g}$
(resp. $(\operatorname{rad})^{u}$). It is interesting to describe
$(\operatorname{rad})^{g}$ and $(\operatorname{rad})^{u}$.

\subsection{TRs connected with the Jacobson radical}

Following the standard notation, we denote the restriction of
$\operatorname{rad}$ to $(BA)$ by $\operatorname{Rad}$. In the Banach algebra
theory, $\operatorname{Rad}$-semisimple (resp. $\operatorname{Rad}$-radical)
algebras are traditionally called \textit{semisimple }(resp. \textit{radical}%
). Now we consider two HTRs on $(NA)$ that coincide with $\operatorname{rad}$
on more wide varieties of normed algebras.

Given an algebra $A\in(NA)$, let $\operatorname{rad}_{b}(A)$ (resp.
$\operatorname{rad}_{n}(A)$) be the intersection of the kernels of all
representation in $\operatorname{irr}_{b}(A^{1})$ (resp. $\operatorname{irr}%
_{n}(A^{1})$). It was shown in \cite[Theorem 10.5]{D97} that
$\operatorname{rad}_{b}$ is a strong HTR on $(NA)$ and all representations in
$\operatorname{irr}_{b}(A)$ are strictly dense. We show the same for
$\operatorname{rad}_{n}$. The following lemma generalizes \cite[Theorem
10.2]{D97}.

\begin{lemma}
\label{p30} For a normed space $X$, every strictly irreducible algebra
$A\subset B(X)$ is strictly dense.
\end{lemma}

\begin{proof}
The proof of \cite[Corollary 1.2.5.4]{B67} given for a Banach space is valid
for a normed space.
\end{proof}

As a consequence, any strictly irreducible representation of a (not
necessarily topological) algebra by bounded operators on a normed space is
strictly dense.

\begin{theorem}
\label{1.7} $(\operatorname{i})$ $\operatorname{rad}_{n}$ is a strong HTR on
$(NA)$.

\begin{itemize}
\item [(ii)]For every $A\in(NA)$, $\operatorname{rad}_{n}(A)$ contains every
(one-sided, not necessarily closed) topologically nil ideal of $A$.
\end{itemize}
\end{theorem}

\begin{proof}
(i) Let $I$ be an ideal of a normed algebra $A$. The equality
$\operatorname{rad}_{n}(I)=I\cap\operatorname{rad}_{n}(A)$ follows immediately
from Lemma \ref{p19}. Setting $I=\operatorname{rad}_{n}(A)$, we deduce
$\operatorname{rad}_{n}(\operatorname{rad}_{n}(A))=\operatorname{rad}_{n}(A)$.
Let again $I=\operatorname{rad}_{n}(A)$, and set $B=A/I$. It follows from
Lemma \ref{p18} that if $b\in\operatorname{rad}(B)$ and $a\in b$ then $a\in
I$, whence $b=0$. We proved that $B$ is $\operatorname{rad}_{n}$-semisimple.

It remains now to prove the equality $f(\operatorname{rad}_{n}(A))\subset
\operatorname{rad}_{n}(B)$, for any continuous epimorphism $f:A\rightarrow B$.
But this immediately follows from the fact that $\pi f\in\operatorname{irr}%
_{n}A$ for every $\pi\in\operatorname{irr}_{n}B$.

(ii) Let $I$ be a topologically nil right ideal. Suppose that, for $\pi
\in\operatorname{irr}_{n}A$, $b\in I$ and $x\in X_{\pi}$, $y=\pi(b)x\neq0$.
There exists $a\in A$ such that $\pi(a)y=x$, whence $\pi(ba)y=y$ and
\[
\left\|  y\right\|  ^{1/n}=\left\|  \pi((ba)^{n})y\right\|  ^{1/n}%
\leqslant(\left\|  \pi\right\|  \left\|  y\right\|  )^{1/n}\left\|
(ba)^{n}\right\|  ^{1/n}\rightarrow0
\]
as $n\rightarrow\infty$, since $ba\in J$. So $y=0$, a contradiction. We obtain
that $\pi J=0$ for every $\pi\in\operatorname{irr}_{n}A$.

If $I$ is a topologically nil left ideal, we use $\pi(ab)x=x$ with the same argument.
\end{proof}

\begin{theorem}
\label{p1} $\operatorname{rad}_{b}=\operatorname{rad}$ on $(Q_{b}A)$ and
$\operatorname{rad}_{n}=\operatorname{rad}$ on $(QA)$.
\end{theorem}

\begin{proof}
It is immediate (since equivalent representations have the same kernels).
\end{proof}

Since $\operatorname{rad}(A)$ is a quasi-regular ideal, then, for a normed
algebra $A$, $\operatorname{rad}(A)$ is a topologically nil ideal of $A$. As a
consequence, we obtain the following well-known assertion.

\begin{corollary}
\label{p3} For an algebra $A\in(QA)$, $\operatorname{rad}A$ is the largest
(one-sided; two-sided) topologically nil ideal of $A$. Moreover, the closure
of a sum of topologically nil ideals of $A$ is a topologically nil ideal of
$A$.
\end{corollary}

\begin{proof}
The statement follows from Theorems \ref{1.7} and \ref{p1} (see also
\cite{P94}).
\end{proof}

\subsection{Uniform TRs}

An $\mathcal{A}$-radical $R$ is called \textit{uniform} if all subalgebras of
an $R$-radical algebra that belong to $\mathcal{A}$ are $R$-radical.

It follows that $\operatorname{rad}$ is a uniform HTR on $(QA)$. It is also
`non-unital' in the sense that no unital algebra can be radical. We show that
$\operatorname{rad}$ is the largest TR of $(QA)$-radicals that share these
properties. Note that, for a commutative normed algebra $A$ and a
representation $\pi\in\operatorname{irr}_{n}A$, $\pi A$ is a normed division
algebra and, by the Gelfand-Mazur Theorem, is one-dimensional.

\begin{proposition}
\label{2.4} If $R$ is a uniform non-unital $(QA)$-radical then $R(A)\subset
\operatorname{rad}(A)$, for every $A\in(QA)$.
\end{proposition}

\begin{proof}
Let us prove firstly that if a $Q$-algebra $A$ is $R$-radical then it is
radical (i.e. $A=\operatorname{rad}(A)$). If not then $A$ contains a
non-quasinilpotent element and therefore there is a maximal non-radical
commutative subalgebra $B$ of $A$. Since clearly $B$ is a spectral subalgebra
of $A$, $B$ is a $Q$-algebra. There is a maximal ideal $I$ of $B$ with
one-dimensional $B/I$ ($I$ $=\ker\pi$ for some $\pi\in\operatorname{irr}_{n}%
B$). It follows that $B/I$ is unital and is not $R$-radical. Then $B$ is not
$R$-radical that contradicts the assumption of uniformity of $R$.

Now for arbitrary $A\in(QA)$ one has that $R(A)$ is radical and
\[
R(A)=\operatorname{rad}(R(A))\subset\operatorname{rad}(A).
\]
\end{proof}

It should be noted that the same statement with actually the same proof holds
for $(BA)$-radicals. In Section \ref{S4} we consider examples of topological
radicals that are uniform on $(NA)$.

\subsection{Regular TRs}

We touch the important problem of extending of a TR to a wider ground class,
in particular from $(BA)$ to $(NA)$. The following simple construction gives
one of possible solutions for hereditary radicals.

Let $R$ be a map on a ground class $\mathcal{A}\supset(BA)$. For $A\in(NA)$
set
\[
R^{\prime}(A)=A\cap R(\overline{A}).
\]
We call $R^{\prime}$ the \textit{regular extension} of $R$ to $(NA)$; $R$ is
called \textit{regular }on $\mathcal{A}$ if $R=R^{\prime}$ on $\mathcal{A}$.

\begin{theorem}
\label{2.6}Let $\mathcal{A}$ be a ground class and $(BA)\subset\mathcal{A}$.
If $R$ is a TR on $\mathcal{A}$ then $R^{\prime}$ is an OTR on $(NA)$; if $R$
is an HTR on $\mathcal{A}$ then so is $R^{\prime}$ on $(NA)$.
\end{theorem}

\begin{proof}
Clearly $R^{\prime}(A)$ is a closed ideal in $A$, for every normed algebra $A$.

Let us firstly prove that, for an ideal $I$ of $A$, $R^{\prime}(I)$ is an
ideal of $A$. Indeed, if $a\in R^{\prime}(I)$ and $b\in A$ then $ab\in I$ and
$ab\in R(\overline{I})$, since $R(\overline{I})$ is an ideal of $\overline{A}%
$. So $ab\in R^{\prime}(I)$ and, similarly, $ba\in R^{\prime}(I)$. We have
\begin{equation}
R^{\prime}(I)=I\cap R(\overline{I})\subset I\cap R(\overline{A})=I\cap A\cap
R(\overline{A})=I\cap R^{\prime}(A), \label{g2}%
\end{equation}
so (5$^{\circ}$) is proved.

For (2$^{\circ}$), set $I=R^{\prime}(A)$. Identifying $A/I$ with
$q_{\overline{I}}(A)$ (for $q_{\overline{I}}:\overline{A}\rightarrow
\overline{A}/\overline{I}$), we identify $\overline{A/I}$ with $\overline
{A}/\overline{I}$. Hence
\[
R^{\prime}(A/I)=q_{\overline{I}}(A)\cap R(\overline{A}/\overline{I}).
\]
Since $\overline{I}$ is $R$-radical by Corollary \ref{2.2}(i), we obtain by
Corollary \ref{2.2}(ii) that $R(\overline{A}/\overline{I})=q_{\overline{I}%
}(R(\overline{A}))$. Thus
\[
R^{\prime}(A/I)=q_{\overline{I}}(A)\cap q_{\overline{I}}(R(\overline
{A}))=q_{\overline{I}}(A\cap R(\overline{A}))=q_{\overline{I}}(I)=0.
\]
The second equality in the above chain follows from the inclusion
$\overline{I}\subset R(\overline{A})$. Indeed if $q_{\overline{I}}(x)\in
q_{\overline{I}}(A)\cap q_{\overline{I}}(R(\overline{A}))$, then there are
elements $a,b\in\overline{I}$ such that $x+a\in A$, $x+b\in R(\overline{A})$.
Hence $x+a\in R(\overline{A})+\overline{I}=R(\overline{A})$ and $q_{\overline
{I}}(x)=q_{\overline{I}}(x+a)\in q_{\overline{I}}(A\cap R(\overline{A}))$.

The property (3$^{\circ}$) is evident, and it remains only to prove
(4$^{\circ}$): we have to show that
\[
q_{I}(R^{\prime}(A))\subset R^{\prime}(A/I)
\]
for a closed ideal $I$ of $A$. Identifying $q_{I}(A)$ and $q_{I}(R^{\prime
}(A))$ with $q_{\overline{I}}(A)\subset\overline{A}/\overline{I}$ and
$q_{\overline{I}}(R^{\prime}(A))$ respectively, we obtain that
\begin{align*}
q_{I}(R^{\prime}(A))  &  =q_{\overline{I}}(R^{\prime}(A))=q_{\overline{I}%
}(R(\overline{A})\cap A)\subset q_{\overline{I}}(R(\overline{A}))\cap
q_{\overline{I}}(A)\\
&  \subset R(\overline{A}/\overline{I})\cap q_{\overline{I}}(A)=R^{\prime
}(A/I).
\end{align*}
We proved that $R^{\prime}$ is an OTR on $(NA)$.

Now let $R$ be an HTR. Then $R^{\prime}(I)=I\cap R^{\prime}(A)$ for every
ideal $I$ of an algebra $A$. Indeed, since $R(\overline{I})=R(\overline
{A})\cap\overline{I}$, we have%
\[
R^{\prime}(I)=R(\overline{I})\cap I=R(\overline{A})\cap\overline{I}\cap
I=R(\overline{A})\cap A\cap I=R^{\prime}(A)\cap I.
\]
Applying the proved equality to $I=R^{\prime}(A)$ we get (1$^{\circ}$):
\[
R^{\prime}(R^{\prime}(A))=R^{\prime}(A).
\]
Therefore $R^{\prime}$ is an HTR on $(NA)$.
\end{proof}

Note that in general $R^{\prime}$ is not a strong TR on $(NA)$ even if $R$ is
a strong HTR on $(BA)$ \cite[Remark 10.8]{D97}. Let us say that an
$(NA)$-radical $R$ is \textit{semi-regular} if $R(A)\subset R(\overline{A}).$
Then clearly the regular extension is the largest semi-regular extension for a
given TR.

Following \cite{D97}, let $T_{\infty}(A)$ (resp. $T_{m}(A)$, for
$m\in\mathbb{N}$) be the intersection of all continuous strongly dense (resp.
topologically $m$-transitive) representations (of $A^{1}$) on Banach spaces.
Recall that a continuous representation $\tau$ of $A$ on a Banach space $X$ is
called \textit{topologically }$m$\textit{-transitive } if the map
$a\longmapsto(\tau(a)x_{1},\ldots\tau(a)x_{m})$, $a\in A$, has a dense image
in the direct sum $X^{m}$ for every linearly independent system $x_{1},\ldots
x_{m}\in X$, and \textit{strongly dense }if this property holds for every
$m\in\mathbb{N}$. Note that $T_{\infty}$ and $T_{m}$ are strong HTRs on $(NA)$
\cite[Theorem 8.1]{D97}, for every $m$.

\begin{proposition}
$T_{\infty}$ and $T_{m}$ are regular HTR on $(NA)$, for every $m$.
\end{proposition}

\begin{proof}
For a normed algebra $A$, the maps $\pi\longmapsto\overline{\pi}$, the
continuous extension of $\pi$ of the algebra $A$ to $\overline{A}$, and
$\tau\longmapsto\tau|_{A}$, the restriction of $\tau$ of the algebra
$\overline{A}$ to $A$, determine bijection between the sets of continuous
strongly dense (resp. topologically $m$-transitive) representations of $A$ and
of $\overline{A}$, respectively. Hence it is easy to see that $T_{\infty
}=T_{\infty}^{\prime}$ and $T_{m}=T_{m}^{\prime}$ for every $m$.
\end{proof}

Clearly a Banach algebra is radical iff it is topologically nil. It follows
that
\[
\operatorname{Rad}^{\prime}(A)=\{a\in A:\rho(ab)=0,\quad\forall b\in
\overline{A}\}
\]
for every $A\in(NA)$. So $\operatorname{Rad}^{\prime}(A)$ is a closed
topologically nil ideal of $A$, and, for every $A\in(QA)$, $\operatorname{Rad}%
^{\prime}(A)\subset\operatorname{rad}(A)$ by Corollary \ref{p3}.

\subsection{\label{S2.11}Universal envelopes}

Now we find out a natural universal class of normed algebras on which
$\operatorname{rad}_{b}$, $\operatorname{rad}_{n}$ and $\operatorname{Rad}%
^{\prime}$ coincide. The comparison of these radicals on the different classes
will be done in the next subsection.

Let $\mathcal{A}$ be a class of normed algebras. Let $\mathcal{A}^{g}$ (resp.
$\mathcal{A}^{u}$) be the smallest ground (resp. universal) class containing
$\mathcal{A}$. We call $\mathcal{A}^{g}$ (resp. $\mathcal{A}^{u}$) the
\textit{ground} (resp. \textit{universal}) \textit{envelope} of $\mathcal{A}$.
Also, let $\mathcal{A}^{i}$ (resp. $\mathcal{A}^{di}$) be the class of all
normed algebras topologically isomorphic to ideals (resp. dense ideals) of
algebras in $\mathcal{A}$.

\begin{lemma}
\label{p6}Let $\mathcal{A}$ be a ground class. Then so is $\mathcal{A}^{i}$
and $\mathcal{A}^{i}=\mathcal{A}^{di}$.
\end{lemma}

\begin{proof}
For $\mathcal{A}^{i}\subset\mathcal{A}^{di}$ it suffices to note that if
$A\in\mathcal{A}^{i}$ then $A$ is identified with an ideal of some algebra
$B\in\mathcal{A}$. Let $\widetilde{A}$ be the closure of $A$ in $B$. Then
$\widetilde{A}$ is a closed ideal of $B$ (hence $\widetilde{A}\in\mathcal{A}$)
and $A$ is an ideal of $\widetilde{A}$. Therefore $A\in\mathcal{A}^{di}$. This
shows that $\mathcal{A}^{i}\subset\mathcal{A}^{di}$, and the converse is evident.

Now let $I$ be a closed ideal of $A$. Let $\widetilde{I}$ be the closure of
$I$ in $\widetilde{A}$. Then $\widetilde{I}$ is a closed ideal of
$\widetilde{A}$ so that $\widetilde{I}\in\mathcal{A}$, and $I=\widetilde
{I}\cap A$. Note that $\widetilde{I}A\cup A\widetilde{I}\subset\widetilde
{I}\cap A=I$, also
\[
I\widetilde{I}=(\widetilde{I}\cap A)\widetilde{I}\subset\widetilde
{I}\widetilde{I}\cap A\widetilde{I}\subset\widetilde{I}\cap A=I
\]
and, similarly, $\widetilde{I}I\subset I$, i.e., $I$ is an ideal of
$\widetilde{I}\in\mathcal{A}$, whence $I\in\mathcal{A}^{i}$.

It remains to show that $A/I\in\mathcal{A}^{i}$. But it is easy: $A/I$ is
isometrically isomorphic to a dense ideal of $\widetilde{A}/\widetilde{I}%
\in\mathcal{A}$, whence $A/I\in\mathcal{A}^{i}$.
\end{proof}

Let $\mathcal{A}$ be a class of normed algebras. Put $\mathcal{A}%
^{(0)}=\mathcal{A}^{g}$ and $\mathcal{A}^{(m+1)}=(\mathcal{A}^{(m)})^{i}$ for
$m=0,1,\ldots$.

\begin{theorem}
\label{p7}$\mathcal{A}^{u}=\cup_{m\geqslant0}\mathcal{A}^{(m)}$.
\end{theorem}

\begin{proof}
It is clear that $\mathcal{A}^{u}$ contains $\mathcal{A}^{(0)}=\mathcal{A}%
^{g}$, and if $\mathcal{A}^{u}$ contains $\mathcal{A}^{(m)}$ then it contains
$\mathcal{A}^{(m+1)}$. Therefore $\mathcal{A}^{u}\supset\cup\mathcal{A}^{(m)}%
$. To the converse it suffices to show that $\cup\mathcal{A}^{(m)}$ is
universal. Since every $\mathcal{A}^{(m)}$ is ground then $\cup\mathcal{A}%
^{(m)}$ is a ground class by Proposition \ref{p4}. Further, if $A\in
\cup\mathcal{A}^{(m)}$, say if $A\in\mathcal{A}^{(k)}$ for some $k$, then
every ideal of $A$ is contained in $\mathcal{A}^{(k+1)}$. So $\cup
\mathcal{A}^{(m)}$ is universal.
\end{proof}

\begin{theorem}
\label{p8}Let $R_{1}$ and $R_{2}$ be HTRs on a universal class $\mathcal{B}$.
If $R_{1}=R_{2}$ on a ground class $\mathcal{A}\subset\mathcal{B}$ then
$R_{1}=R_{2}$ on $\mathcal{A}^{u}$.
\end{theorem}

\begin{proof}
It is clear that $\mathcal{A}^{u}\subset\mathcal{B}$. Note that $R_{1}=R_{2}$
on $\mathcal{A}^{(0)}$. Suppose that $R_{1}=R_{2}$ on $\mathcal{A}^{(k)}$. If
$A\in\mathcal{A}^{(k+1)}$ then $A$ is identified with an ideal of some
$B\in\mathcal{A}^{(k)}$. Since $R_{1}$ and $R_{2}$ are HTRs on $\mathcal{A}%
^{u}$, we have
\[
R_{1}(A)=R_{1}(B)\cap A=A\cap R_{2}(B)=R_{2}(A).
\]
Therefore $R_{1}=R_{2}$ on $\mathcal{A}^{(k+1)}$, hence on $\cup
\mathcal{A}^{(m)}=\mathcal{A}^{u}$.
\end{proof}

\begin{corollary}
\label{p9}$\operatorname{Rad}^{\prime}=\operatorname{rad}_{n}%
=\operatorname{rad}_{b}=\operatorname{rad}$ on $(BA)^{u}$.
\end{corollary}

\begin{proof}
Recall that these radicals are HTRs on $(QA)$ and clearly $(BA)^{u}%
\subset(QA)$. Apply Theorem \ref{p8}.
\end{proof}

Note that $(BA)^{u}\subset(Q_{b}A)$ because $(Q_{b}A)$ is a universal class
containing $(BA)$. We will show in the next subsection that the inclusion is strict.

\subsection{Comparison of the radicals}

Let us write $R_{1}\leqslant R_{2}$ on a class $\mathcal{A}$ if $R_{1}%
(A)\subset R_{2}(A)$, for each $A\in\mathcal{A}$; if, for some $A\in
\mathcal{A}$, the inclusion is strict we write $R_{1}<R_{2}$ on $\mathcal{A}$.

It is clear that $\operatorname{rad}_{n}\leqslant\operatorname{rad}_{b}$ on
$(NA)$; we will prove that they differ already on $(QA)$.

\begin{proposition}
\label{p60}The inclusion $(Q_{b}A)\subset(QA)$ is strict and
$\operatorname{rad}_{n}<\operatorname{rad}_{b}$ on $(QA)$.
\end{proposition}

\begin{proof}
It will be sufficient to construct an algebra $A\in(QA){\setminus}(Q_{b}A)$
with $\operatorname{rad}_{n}(A)\neq\operatorname{rad}_{b}(A)$.

Let $H$ be a separable Hilbert space. Given a basis we identify an operator on
$H$ with a matrix. Let $A$ be the algebra of all matrices with only finite
number of nonzero entries, supplied with the operator norm. Every element of
$A$ generates a finite-dimensional subalgebra, whence $\sum_{n>0}a^{n}$
converges, for $\Vert a\Vert<1$. So $A$ is a $Q$-algebra.

Note that $A$ has no nonzero proper ideals. If no and $I$ is such an ideal,
then $I$ contains an operator with only one nonzero entry and therefore all
operators in $A$, i.e. coincides with $A$, a contradiction.

Clearly $A$ is the union of a sequence of finite-dimensional subalgebras
$A_{n}$. So, for a nonzero representation $\pi$ of $A$ on a Banach space $X$,
$\pi(A)x$ has a countable Hamel basis for every vector $x\in X$ and cannot
coincide with $X$ if $\dim X=\infty$. If $\dim X<\infty$ then $\ker\pi$ is a
nonzero proper ideal of $A$, that is impossible. Hence $A$ has no strictly
irreducible representations on Banach spaces, whence $A=$ $\operatorname{rad}%
_{b}(A)$.

On the other hand, the linear span $H_{0}$ of the basis is invariant for $A$
and the restriction of $A$ to $H_{0}$ is a strictly irreducible
representation. Hence $A\notin(Q_{b}A)$ and $\operatorname{rad}_{n}(A)=0$.
\end{proof}

\begin{proposition}
\label{p61}$\operatorname{Rad}^{\prime}<\operatorname{rad}_{n}$ on $(NA)$;
moreover, $\operatorname{Rad}^{\prime}<\operatorname{rad}$ on $(Q_{b}A)$ and
the inclusion $(BA)^{u}\subset(Q_{b}A)$ is strict.
\end{proposition}

\begin{proof}
Since $\operatorname{Rad}^{\prime}(A)$ is a topologically nil ideal of a
normed algebra $A$, $\operatorname{Rad}^{\prime}\leqslant\operatorname{rad}%
_{n}$ on $(NA)$ by Theorem \ref{1.7}(ii).

Now we show that $\operatorname{Rad}^{\prime}<\operatorname{rad}$ on
$(Q_{b}A)$. Dixon \cite[Example 9.3]{D97} constructed a radical Banach algebra
$A$ and a continuous isomorphism $\phi:A\rightarrow C$ onto a dense subalgebra
$C$ of a semisimple Banach algebra $B$. Then $C$ is radical
($C=\operatorname{rad}(C)$), hence a $Q_{b}$-algebra. On the other hand,
$\operatorname{Rad}^{\prime}(C)=C\cap\operatorname{Rad}(A)=0$.

We see that the inclusion $(BA)^{u}\subset(Q_{b}A)$ is strict because
$\operatorname{Rad}^{\prime}=\operatorname{rad}$ on $(BA)^{u}$ by Corollary
\ref{p9}.
\end{proof}

This indicates that $R^{\prime}$ need not be a maximal extension of an HTR
$R$, and, as a consequence, an HTR on $(NA)$ need not be semi-regular.

\begin{proposition}
$\;T_{\infty}<\operatorname{Rad}^{\prime}$ on $(NA)$.
\end{proposition}

\begin{proof}
Clearly $T_{\infty}\leqslant\operatorname{Rad}^{\prime}$ on $(BA)$ because
strictly irreducible representations of Banach algebras are strictly dense;
the inequality is strict by \cite[Theorem 9.2 and example 9.3]{D97}. By
regularity, $T_{\infty}<\operatorname{Rad}^{\prime}$ on $(NA)$.
\end{proof}

As a consequence of the results, we have the \textit{strict inclusions} in the
following chain of universal classes%
\[
(BA)^{u}\subset(Q_{b}A)\subset(QA).
\]
Let us also write together the obtained strict inequalities for radicals on
$(NA)$:%
\[
T_{\infty}<\operatorname{Rad}^{\prime}<\operatorname{rad}_{n}%
<\operatorname{rad}_{b}.
\]

\subsection{Topologically characteristic and symmetric radicals}

The following notion can be of use in dealing with Lie subalgebras of normed algebras.

A TR $R$ is called \textit{topologically characteristic} on a class
$\mathcal{A}$ if, for every $A\in\mathcal{A}$ and every bounded derivation $D$
of $A$, $DR(A)\subset R(A)$.

\begin{lemma}
All $(BA)$-radicals are topologically characteristic.
\end{lemma}

\begin{proof}
Let $R$ be a TR on $(BA)$. If $D$ is a bounded derivation on a Banach algebra
$A$ then $\exp(\lambda D)$ is an automorphism of $A$ for every $\lambda
\in\mathbb{C}$, whence
\[
\exp(\lambda D)R(A)=R(A).
\]
Since $R(A)$ is closed in $A$,
\[
\lim_{\lambda\rightarrow0}(\exp(\lambda D)a-a)/\lambda\in R(A)
\]
for every $a\in R(A)$, whence $D(a)\in R(A)$.
\end{proof}

\begin{theorem}
\label{p100}Let $\mathcal{A}$ be a class such that $(BA)\subset\mathcal{A}$.
Any regular TR $R$ on $\mathcal{A}$ is topologically characteristic.
\end{theorem}

\begin{proof}
Indeed, for every bounded derivation $D$ on an algebra $A\in$ $\mathcal{A}$,
we have
\[
DR(A)=D(A\cap R(\overline{A}))\subset DA\cap DR(\overline{A})\subset A\cap
R(\overline{A})=R(A).
\]
\end{proof}

In particular $\operatorname{Rad}^{\prime}$, $T_{\infty}$ and $T_{m}$ are
topologically characteristic on $(NA)$, for every $m$.

\begin{theorem}
$\operatorname{rad}_{b}$ and $\operatorname{rad}_{n}$ are topologically
characteristic on $(NA)$.
\end{theorem}

\begin{proof}
Let $D$ be a bounded derivation on an algebra $A\in(NA)$, and let
$a\in\operatorname{rad}_{b}(A)$ (resp. $\operatorname{rad}_{n}(A)$) be
arbitrary. By \cite[Lemma 2.1]{S69}, $\pi(Da)$ is quasinilpotent for every
$\pi\in\operatorname{irr}_{b}A$ (resp. $\operatorname{irr}_{n}A$). If
$y=\pi(Da)x\neq0$ for some $\pi$ and $x\in X_{\pi}$, then there is an element
$b\in A$ such that $\pi(b)y=x$, Hence
\[
\pi(D(ab))y=\pi(Da)\pi(b)y+\pi(a)\pi(Db)y=\pi(Da)x=y
\]
and also $ab\in\operatorname{rad}_{b}(A)$ (resp. $\operatorname{rad}_{n}(A)$),
a contradiction. Therefore $\operatorname{rad}_{b}$ and $\operatorname{rad}%
_{n}$ are topologically characteristic on $(NA)$.
\end{proof}

Many TRs considered were defined in an asymmetric way: by using
representations. For instance the `left-defined' TRs are $\operatorname{rad}%
_{n}$, $\operatorname{rad}_{b}$, $T_{\infty}$ etc. On the other hand, the
using anti-representations for definition of similar,`right-defined' radicals
is also a right way. So, we consider the opposite TRs for asymmetric TRs in
the general setting.

Let $A^{\operatorname{op}}$ be the same algebra $A$, but with the opposite
multiplication. Recall that $A^{\operatorname{op}}$ is called an
\textit{opposite }algebra. The `identity' anti-isomorphism $a\longmapsto
a^{\operatorname{op}}$ maps an element $a$ to the same element but in
$A^{\operatorname{op}}$; note that $a^{\operatorname{op}}b^{\operatorname{op}%
}=(ba)^{\operatorname{op}}$ for all $a,b\in A$ and every homomorphism
$f:A\rightarrow B$ induces the \textit{opposite} homomorphism
$f^{\operatorname{op}}:A^{\operatorname{op}}\rightarrow B^{\operatorname{op}}$
by formula $f^{\operatorname{op}}(a^{\operatorname{op}}%
)=(f(a))^{\operatorname{op}}$, for every $a\in A$. Clearly
$(f^{\operatorname{op}})^{\operatorname{op}}=f$.

Let $\mathcal{A}\subset(NA)$, and let $\mathcal{A}^{\operatorname{op}}$ be the
class of all algebras $A$ such that $A^{\operatorname{op}}\in\mathcal{A}$.
Clearly $(\mathcal{A}^{\operatorname{op}})^{\operatorname{op}}=\mathcal{A}$;
we call $\mathcal{A}^{\operatorname{op}}$ an \textit{opposite class.} All
considered above properties of a class are inherited by the opposite class. A
class $\mathcal{A}$ is called \textit{symmetric }if $\mathcal{A}%
=\mathcal{A}^{\operatorname{op}}$. Note that $(BA)$ and $(QA)$ are symmetric.
Moreover, all closed subalgebras of a $Q$-algebra are $Q$-algebras. It is not
clear for $Q_{b}$-algebras. Is $(QA_{b})$ symmetric? Does it contain all
closed subalgebras of its algebras?

Let $\mathcal{A}$ be a ground class. Then clearly $\mathcal{A}%
^{\operatorname{op}}$ is a ground class. If $R$ is defined on $\mathcal{A}$,
one can define $R^{\operatorname{op}}$ on $\mathcal{A}^{\operatorname{op}}$
by
\[
R^{\operatorname{op}}(A^{\operatorname{op}})=R(A)^{\operatorname{op}}%
\]
for every $A\in\mathcal{A}$. We call $R^{\operatorname{op}}$ to be
\textit{opposite }to $R$. If $\mathcal{A}$ is symmetric, $R$ is called
\textit{symmetric }on $\mathcal{A}$ if $R(A)=R^{\operatorname{op}}(A)$ for
every $A\in\mathcal{A}$.

\begin{proposition}
Let $R$ be a TR or HTR on a ground class $\mathcal{A}$. Then so is
$R^{\operatorname{op}}$ on $\mathcal{A}^{\operatorname{op}}$.
\end{proposition}

\begin{proof}
For (1$^{\circ}$), we have
\[
R^{\operatorname{op}}(A^{\operatorname{op}})=R(A)^{\operatorname{op}%
}=R(R(A))^{\operatorname{op}}=R^{\operatorname{op}}(R(A)^{\operatorname{op}%
})=R^{\operatorname{op}}(R^{\operatorname{op}}(A^{\operatorname{op}})).
\]
The easy checkup of the other properties is left to the reader.
\end{proof}

Note that $\operatorname{Rad}^{\prime}$ is symmetric on $(NA)$ and
$\operatorname{rad}_{n}$ on $(QA)$. Are $\operatorname{rad}_{n}$ and
$\operatorname{rad}_{b}$ symmetric on $(NA)$?

\section{\label{S3}Tensor radicals}

In this section we consider the behavior of a topological radical with respect
to the fundamental operations: direct sum ${{\oplus}}$ and projective tensor
product ${\widehat{\otimes}}$ of algebras.

\subsection{Radicals on direct sums}

Actually for direct sums the problem is easy.

\begin{theorem}
\label{sum} Let $\mathcal{A}$ be a ground class. If $R$ is a TR on
$\mathcal{A}$ and $A{\oplus}B\in\mathcal{A}$ for some $A,B\in\mathcal{A}$ then
$R(A{\oplus}B)=R(A){\oplus}R(B)$.
\end{theorem}

\begin{proof}
Let $f_{1}$ and $f_{2}$ be the epimorphisms of $A{\oplus}B$ onto $A$ and $B$
respectively, defined as the natural projections (which are open and
continuous). Then by (\ref{g1}) we have that projections of $R(A{\oplus}B)$
are contained in $R(A)$ and $R(B)$, respectively. Hence
\[
R(A{\oplus}B)\subset R(A){\oplus}R(B).
\]
On the other hand if one considers the ideals $I_{1}=A{\oplus}0$ and
$I_{2}=0{\oplus}B$ and applies (5$^{\circ}$) then the inclusions $R(A{\oplus
}0)\subset R(A{\oplus}B)$ and $R(0{\oplus}B)\subset R(A{\oplus}B)$ will be
established. But $f_{1}$ defines a topological isomorphism of $A{\oplus}0$
onto $A$, whence (3$^{\circ}$) gives $R(A{\oplus}0)=R(A){\oplus}0$ and,
similarly, $R(0{\oplus}B)=0{\oplus}R(B)$. So
\[
R(A){\oplus}R(B)\subset R(A{\oplus}B)
\]
and we are done.
\end{proof}

\subsection{Tensor and weakly tensor radicals on $(BA)$}

The second question is much more difficult. We will consider it only for
topological radicals on $(BA)$ because the projective tensor product is a
Banach space operation, i.e. for $A$ and $\overline{A}$ the result is the same.

Let us denote by $A{{\otimes}}B$ the algebraic tensor product of Banach
algebras $A$ and $B$. Clearly $A{{\otimes}}B$ can be considered as a linear
manifold of $A{\widehat{\otimes}}B$. For every subsets $M\subset A$ and
$N\subset B$ that are not linear manifolds, it is convenient to denote by
$M{{\otimes}}N$ the set $\{a\otimes b\in A\otimes B:a\in M$, $b\in N\}$. If
one of them is a linear manifold, let $M\otimes N$ be the linear span in
$A{{\otimes}}B$ of all elements $a{{\otimes}}b$, $a\in M$, $b\in N$. In any
case, $M{\otimes}N$ can be also considered as a subset of $A{\widehat{\otimes
}}B$. Let $\widetilde{M{{\otimes}}N}$ be denote the closure of $M{\otimes}N$
in $A{\widehat{\otimes}}B$.

A topological radical $R$ on $(BA)$ is called \textit{tensor} if
\[
R(A){\otimes}B\subset R(A{\widehat{\otimes}}B),
\]
for every Banach algebras $A,B$.

For a wide class of TRs this condition admits a convenient reformulation.

\begin{theorem}
\label{istrong} An ideally strong radical $R$ on $(BA)$ is tensor if and only
if the tensor product of an $R$-radical algebra and arbitrary Banach algebra
is $R$-radical.
\end{theorem}

The proof will be given after some preliminary work.

If $J$ is a closed ideal in a Banach algebra $A$ then, for any Banach algebra
$B$, the `identity' map $i_{A}:J{\widehat{\otimes}}B\rightarrow A{\widehat
{\otimes}}B$ is contracting. The following lemma is probably known, but we
could not find a precise reference.

\begin{lemma}
\label{ideal} Let $A,B\in(BA)$, and let $J$ be a closed ideal of $A$.

\begin{itemize}
\item [(i)]$i_{A}(J{\widehat{\otimes}}B)$ is a two-sided ideal
(non-necessarily closed) of $A{\widehat{\otimes}}B$ consisting of all elements
that can be represented in the form $\sum a_{n}{\otimes}b_{n}$ with $a_{n}\in
J$ and $\sum\Vert a_{n}\Vert\Vert b_{n}\Vert<\infty$.

\item[(ii)] There exists a unique contractive epimorphism $\tau_{J}%
:(A/J){\widehat{\otimes}}B\rightarrow(A{\widehat{\otimes}}B)/U$, where $U$ is
the closure of $i_{A}(J{\widehat{\otimes}}B)$ in $A{\widehat{\otimes}}B$, such
that
\[
\tau_{J}((a+J){\otimes}b)=a{\otimes}b+U
\]
for every $a\in A$, $b\in B$.
\end{itemize}
\end{lemma}

\begin{proof}
(i) is evident.

(ii) Let $E=A{\widehat{\otimes}}B/U$ and as usually $q_{U}:A{\widehat{\otimes
}}B\rightarrow E$ the canonical epimorphism. We define a linear map
$\tau:(A/J){\otimes}B\rightarrow E$ by
\[
\tau(\sum(a_{n}+J){\otimes}b_{n})=q_{U}(\sum a_{n}{\otimes}b_{n}).
\]
To justify the definition, note firstly that it does not depend on the choice
of representatives: if $a_{n}+J=a_{n}^{\prime}+J$ for every $n$ then $\sum
a_{n}{\otimes}b_{n}-\sum a_{n}^{\prime}{\otimes}b_{n}\in J{\otimes}B$ and
therefore%
\[
q_{U}(\sum a_{n}{\otimes}b_{n})=q_{U}(\sum a_{n}^{\prime}{\otimes}b_{n}).
\]
So it suffices to check the bilinearity which is easy.

Now we have to prove that $\tau$ is contractive. For any $T\in(A/J){\otimes}B$
and any $\varepsilon>0$ there exist $a_{n}\in A$, $b_{n}\in B$ ($n=1,...,m$)
such that $T=\sum(a_{n}+J){\otimes}b_{n}$, $\Vert b_{n}\Vert=1$ and
\[
\sum\left\|  a_{n}+J\right\|  <\left\|  T\right\|  +\varepsilon.
\]
Choosing $a_{n}^{\prime}\in a_{n}+J$ with
\[
\left\|  a_{n}^{\prime}\right\|  <\left\|  a_{n}+J\right\|  +\varepsilon/m,
\]
we get that
\begin{align*}
\left\|  \tau(T)\right\|   &  =\left\|  q_{U}(\sum a_{n}^{\prime}{\otimes
}b_{n})\right\|  \leqslant\left\|  \sum a_{n}^{\prime}{\otimes}b_{n}\right\|
\leqslant\left\|  \sum a_{n}^{\prime}\right\| \\
&  <\left\|  T\right\|  +2\varepsilon.
\end{align*}
Since $\varepsilon$ is arbitrary, $\left\|  \tau(T)\right\|  \leqslant\left\|
T\right\|  $. The linearity and multiplicativity of $\tau$ are evident.

Denote by $\tau_{J}$ the contractive homomorphism of $(A/J){\widehat{\otimes}%
}B$ to $E$ that extends $\tau$ by continuity. Then
\[
\tau_{J}(\sum_{1}^{\infty}(a_{n}+J){\otimes}b_{n})=q_{U}(\sum_{1}^{\infty
}a_{n}{\otimes}b_{n})
\]
whenever $\sum\left\|  a_{n}\right\|  \left\|  b_{n}\right\|  <\infty$. Hence
$\tau_{J}$ is surjective.
\end{proof}

\begin{lemma}
\label{equiv} A topological radical $R$ on $(BA)$ is tensor iff
\[
i_{A}(R(A){\widehat{\otimes}}B)\subset R(A{\widehat{\otimes}}B),
\]
for all Banach algebras $A,B$.
\end{lemma}

\begin{proof}
Follows from the evident inclusions
\[
R(A){\otimes}B\subset i_{A}(R(A){\widehat{\otimes}}B)\subset\widetilde
{R(A){\otimes}B}%
\]
(the latter is as usually the closure of $R(A){\otimes}B$ in $A{\widehat
{\otimes}}B$) and the fact that $R(A{\widehat{\otimes}}B)$ is closed.
\end{proof}

\begin{proof}
[\textit{The proof of Theorem \ref{istrong}}]If $R$ is tensor and $A$ is
$R$-radical then
\[
A{\otimes}B=R(A){\otimes}B\subset R(A{\widehat{\otimes}}B),
\]
whence $A{\widehat{\otimes}}B\subset R(A{\widehat{\otimes}}B)$ and accordingly
$A{\widehat{\otimes}}B$ is $R$-radical.

Conversely, let $R$ be ideally strong and suppose that the tensor product of
an $R$-radical algebra and arbitrary Banach algebra is $R$-radical. Let
$A,B\in(BA)$ be arbitrary. The map $i_{A}:R(A){\widehat{\otimes}}B\rightarrow
A{\widehat{\otimes}}B$ is a continuous epimorphism onto the ideal
$i_{A}(R(A){\widehat{\otimes}}B)$ of $A{\widehat{\otimes}}B$ (see Lemma
\ref{ideal}) and the algebra $R(A){\widehat{\otimes}}B$ is $R$-radical by our
assumptions. Hence
\[
i_{A}(R(A){\widehat{\otimes}}B)=i_{A}(R(R(A){\widehat{\otimes}}B))\subset
R(A{\widehat{\otimes}}B).
\]
Using Lemma \ref{equiv}, we conclude that $R$ is tensor.
\end{proof}

We will say that a $(BA)$-radical $R$ is \textit{weakly tensor} if
$A{\widehat{\otimes}}B$ is $R$-radical, for any $R$-radical algebra $A$ and
arbitrary algebra $B$. So Theorem \ref{istrong} states that a weakly tensor,
ideally strong radical is tensor.

Recall that the class $(BA)^{i}$ of all normed algebras topologically
isomorphic to ideals of Banach algebras is a ground class (Lemma \ref{p6}) and
all $(BA)$-radicals are strong. The following simple assertion underlines the
important role of existence of strong extensions of weakly tensor
$(BA)$-radicals to $(BA)^{i}$.

\begin{proposition}
Let $R$ be a weakly tensor TR on $(BA)$. If $R$ admits an extension to
$(BA)^{i}$ as a strong TR then $R$ is tensor.
\end{proposition}

\begin{proof}
Indeed, if $R$ does then $R$ is ideally strong and hence tensor by Theorem
\ref{istrong}.
\end{proof}

Some examples of tensor or weakly tensor radicals we will see in this and the
next sections as well as in subsequent publications of our project.

\subsection{The radical $R^{t}$}

In general, given a TR $R$ on $(BA)$, one can try to construct a related
tensor radical $R^{t}$ in the following way:
\begin{equation}
R^{t}(A)=\{a\in A:a{\otimes}B\subset R(A{\widehat{\otimes}}B),\quad\forall
B\in(BA)\}.
\end{equation}
It follows from the above arguments that if $R^{t}$ is a TR then it is a
tensor one (because of the associativity of tensor product) and that $R$
itself is tensor iff $R=R^{t}$. It is important to know conditions under which
$R^{t}$ is a TR. We will obtain now some results in this direction.

Let $I$ be a closed ideal of a Banach algebra $A$. For a Banach algebra $B$,
the homomorphism $i_{A}:I{\widehat{\otimes}}B\rightarrow A{\widehat{\otimes}%
}B$ is injective if say, $I$ has a bounded approximate identity. Moreover, in
this case $i_{A}$ is bounded from below. Indeed, given a b.a.i. $e_{\lambda}$
in $J$ with $\beta=\sup\Vert e_{\lambda}\Vert$, define a net of operators
$S_{\lambda}$ on $J{\widehat{\otimes}}B$ by $S_{\lambda}(a{\otimes
}b)=e_{\lambda}a{\otimes}b$. Then $S_{\lambda}\rightarrow1$ in the strong
operator topology. Now if $G\in J{\widehat{\otimes}}B$ and $F=i_{A}(G)$ then,
for any presentation of $F$ in the form $F=\sum a_{n}{\otimes}b_{n}$, one has
\[
\left\|  S_{\lambda}G\right\|  \leqslant\sum\left\|  e_{\lambda}a_{n}\right\|
\left\|  b_{n}\right\|  \leqslant\beta\sum\left\|  a_{n}\right\|  \left\|
b_{n}\right\|  .
\]
It follows that $\Vert S_{\lambda}G\Vert\leqslant\beta\Vert F\Vert$ and,
passing to the limit, $\Vert G\Vert\leqslant\beta\Vert F\Vert$.

In general $i_{A}$ can have a non-zero kernel $K=K(A,I,B)$. The algebras
$K(A,I,B)$ will be called \textit{tensor pathological algebras}.

\begin{theorem}
\label{rtrad} Let $R$ be an HTR on $(BA)$. Then

\begin{itemize}
\item [(i)]$R^{t}(A)$ is a closed ideal of $A$, for any Banach algebra $A$.

\item[(ii)] If $R$ is ideally strong then $R^{t}$ satisfies conditions
$(2^{\circ})$, $(3^{\circ})$, $(4^{\circ})$ and $(5_{2}^{\circ})$.

\item[(iii)] If $R$ is strictly hereditary and if all tensor pathological
algebras are $R$-radical then $R^{t}$ is an HTR on $(BA)$.
\end{itemize}
\end{theorem}

\begin{proof}
(i) Let $I=\{a\in A:a{\otimes}B\subset R(A{\widehat{\otimes}}B)$ for every
unital $B\in(BA)\}$. Then $I$ is an ideal of $A$: if $a{\otimes}b\in
R(A{\widehat{\otimes}}B)$ then
\[
(ca){\otimes}b=(c{\otimes}1)(a{\otimes}b)\in R(A{\otimes}B).
\]
Clearly $R^{t}(A)\subset I$. We will prove that $I=R^{t}(A)$.

Note first that if $B$ is a non-unital Banach algebra and $j$ the natural
homomorphism of $A{\widehat{\otimes}}B$ to $A{\widehat{\otimes}}B^{1}$ then
$j$ is isometric and $j(A{\widehat{\otimes}}B)$ is a closed ideal of
$A{\widehat{\otimes}}B^{1}$ (here $A{\widehat{\otimes}}B^{1}$ is the product
of $A$ and $B^{1}$).

Indeed, if $F\in A{\widehat{\otimes}}B$ and $j(F)$ has a representative $\sum
a_{n}{\otimes}(b_{n}+\lambda_{n})$ with $\lambda_{n}$ scalar multiplies of the
unit such that, given $\varepsilon>0$,
\[
\Vert j(F)\Vert+\varepsilon>\sum\left\|  a_{n}\right\|  \left\|  b_{n}%
+\lambda_{n}\right\|  =\sum\left\|  a_{n}\right\|  (\left\|  b_{n}\right\|
+\left|  \lambda_{n}\right|  )
\]
(using the standard norm in $B\oplus\mathbb{C}=B^{1}$) then $\sum\lambda
_{n}a_{n}$ converges to $0$ since $A{\widehat{\otimes}}B$ is a complemented
subspace of $A{\widehat{\otimes}}B^{1}$, whence $F=\sum a_{n}{\otimes}b_{n}$
and
\[
\left\|  j(F)\right\|  +\varepsilon>\sum{\Vert}a_{n}{\Vert\Vert}b_{n}%
{\Vert\geqslant\Vert}F{\Vert}.
\]
It follows that $\Vert j(F)\Vert\geqslant\Vert F\Vert$. The converse
inequality is evident. Hence clearly $j(A{\widehat{\otimes}}B)$ is a closed
ideal of $A{\widehat{\otimes}}B^{1}$.

If $a\in I$ then $a{\otimes}B^{1}\subset R(A{\widehat{\otimes}}B^{1})$,
\[
j(a{\otimes}B)\subset j(A{\widehat{\otimes}}B)\cap R(A{\widehat{\otimes}}%
B^{1})=R(j(A{\widehat{\otimes}}B))=j(R(A{\widehat{\otimes}}B))
\]
(we used (6$^{\circ}$) and (3$^{\circ}$) for $R$), whence $a{\otimes}B\subset
R(A{\widehat{\otimes}}B)$ and $a\in R^{t}(A)$. Thus $I=R^{t}(A)$ is an ideal
of $A$, obviously closed.

(ii) Let $I$ be a closed ideal of a Banach algebra $A$. If $a\in R^{t}(I)$
then, for each $B\in(BA)$ and each $b\in B$, $a{\otimes}b\in R(I{\widehat
{\otimes}}B)$. By Lemma \ref{ideal}, the image of the continuous homomorphism
$i_{A}:I{\widehat{\otimes}}B\rightarrow A{\widehat{\otimes}}B$ is an ideal of
$A{\widehat{\otimes}}B$. Since $R$ is ideally strong, $i_{A}(a{\otimes}b)\in
R(A{\widehat{\otimes}}B)$, that is $a{\otimes}b\in R(A{\widehat{\otimes}}B)$.
This means that $a\in R^{t}(A)$, whence $R^{t}(I)\subset I\cap R^{t}(A)$. We
proved that $R^{t}$ satisfies (5$_{2}^{\circ}$).

Now let us show that $q_{I}(R^{t}(A))\subset R^{t}(A/I)$. Let $a\in R^{t}(A)$.
Then $a{\otimes}b\in R(A{\widehat{\otimes}}B)$ for each $B\in(BA)$. Let
$f:A{\widehat{\otimes}}B\rightarrow(A/I){\widehat{\otimes}}B$ be the
epimorphism defined by $f(x{\otimes}y)=(x+I){\otimes}y$ for every $x\in A$,
$y\in B$. Then
\[
f(a{\otimes}b)\in R((A/I){\widehat{\otimes}}B),\quad(a+I){\otimes}b\in
R((A/I){\widehat{\otimes}}B),
\]

whence $a+I\in R^{t}(A/I),$and therefore $R^{t}$satisfies (4$^{\circ}$).

The property (3$^{\circ}$) is evident and we have to prove (2$^{\circ}$).
Denote $R^{t}(A)$ by $J$ for brevity. Let, as in Lemma \ref{ideal}, $U$ be the
closure of $i_{A}(J{\widehat{\otimes}}B)$ in $A{\widehat{\otimes}}B$ and
$\tau_{J}$ the epimorphism of $(A/J){\widehat{\otimes}}B$ onto $(A{\widehat
{\otimes}}B)/U$. Since $i_{A}(J{\widehat{\otimes}}B)$ is an ideal of
$A{\widehat{\otimes}}B$ (Lemma \ref{ideal}) and $R$ is ideally strong, the
ideal $U$ is $R$-radical. Using Corollary \ref{2.1}, we get that
\begin{equation}
R((A{\widehat{\otimes}}B)/U)=q_{U}(R(A{\widehat{\otimes}}B)). \label{f40}%
\end{equation}
Furthermore, it follows from (7$^{\circ}$) (that holds for $R$) that
\begin{equation}
\tau_{J}(R((A/J){\widehat{\otimes}}B))\subset R((A{\widehat{\otimes}}B)/U).
\label{f41}%
\end{equation}
Therefore, it follows from (\ref{f40}) and (\ref{f41}) that
\[
\tau_{J}(R((A/J){\widehat{\otimes}}B))\subset q_{U}(R(A{\widehat{\otimes}%
}B)).
\]

Let now $a+J\in R^{t}(A/J)$. Then $(a+J){\otimes}b\subset R((A/J){\widehat
{\otimes}}B)$, whence%
\[
\tau_{J}((a+J){\otimes}b)\in q_{U}(R(A{\widehat{\otimes}}B),
\]
and $a{\otimes}b+U\in q_{U}(R(A{\widehat{\otimes}}B))$. Since $U\subset
R(A{\widehat{\otimes}}B)$, we conclude that $a{\otimes}b\in R(A{\widehat
{\otimes}}B)$. Hence $a\in R^{t}(A)=J$ and $a+J=0$. This shows that
(2$^{\circ}$) holds for $R^{t}$.

(iii) Suppose now that $R$ is strictly hereditary and all tensor pathological
algebras are $R$-radical. We need only to prove the property (6$^{\circ}$)
(which implies (1$^{\circ}$) and (5$_{1}^{\circ}$)). Let again $I$ be a closed
ideal of a Banach algebra $A$. Fix $B\in(BA)$ and denote by $K$ the kernel of
the homomorphism $i_{A}:I{\widehat{\otimes}}B\rightarrow A\widehat{{\otimes}%
}B$ and by $h$ the induced homomorphism of $(I{\widehat{\otimes}}B)/K$ to
$A{\widehat{\otimes}}B$. Note that $i_{A}=hq_{K}$. It is important that $h$ is
injective and its image is an ideal of $A{\widehat{\otimes}}B$.

Let now $a\in I\cap R^{t}(A)$. Then, for each $b\in B$, the element
$i_{A}(a{\otimes}b)$ belongs to $R(A{\widehat{\otimes}}B)$ (here we consider
$a{\otimes}b$ as an element of $I{\widehat{\otimes}}B$). Since $R$ is strictly
hereditary,
\[
h(R((I{\widehat{\otimes}}B)/K))=i_{A}(I{\widehat{\otimes}}B)\cap
R(A{\widehat{\otimes}}B),
\]
whence%
\[
i_{A}(a{\otimes}b)\in h(R((I{\widehat{\otimes}}B)/K)).
\]
Since the algebra $K$ is tensor pathological, it is $R$-radical and, by
Corollary \ref{2.1},
\[
R((I{\widehat{\otimes}}B)/K)=q_{K}(R(I{\widehat{\otimes}}B)).
\]
Thus%
\[
i_{A}(a{\otimes}b)\in h(q_{K}(R(I{\widehat{\otimes}}B)))=i_{A}(R(I{\widehat
{\otimes}}B)),
\]
whence
\[
a{\otimes}b-F\in\ker i_{A}%
\]
for some $F\in R(I{\widehat{\otimes}}B)$. Since $\ker i_{A}=K$ is $R$-radical,
it is contained in $R(I{\widehat{\otimes}}B)$, that gives
\[
a{\otimes}b\in R(I{\widehat{\otimes}}B).
\]
Therefore $a\in R^{t}(I)$. We proved the inclusion $I\cap R^{t}(A)\subset
R^{t}(I)$. The converse inclusion was established in (5$_{2}^{\circ}$).
\end{proof}

\subsection{\label{sub}Tensor properties of the Jacobson radical}

We know that the Jacobson radical $\operatorname{Rad}$ is strictly hereditary
(see Corollary \ref{p2}). But the result of Theorem \ref{rtrad} cannot be
immediately applied to $R=\operatorname{Rad}$ because we do not know if the
tensor pathological algebras are radical. The proof of the fact that
$\operatorname{Rad}^{t}$ is an HTR will be finished by means of a technique
related to the notion of \textit{joint quasinilpotence}.

The most interesting and important problem is one of the coincidence
$\operatorname{Rad}^{t}$ with $\operatorname{Rad}$. In the form ``is the
projective tensor product of a radical Banach algebra $A$ and arbitrary Banach
algebra $B$ radical?'' it in fact goes back to \cite{A79}, where the case of
commutative $B$ was solved. Being non-able to solve it in the full generality,
we consider its individual aspects.

Let us call a Banach algebra $A$ \textit{tensor radical} if
$A=\operatorname{Rad}^{t}(A)$ that is if $A{\widehat{\otimes}}B$ is the
Jacobson radical for arbitrary Banach algebra $B$; $A$ is called
\textit{tensor perfect} if $A{\widehat{\otimes}}B$ is radical, for any radical
Banach algebra $B$.

In what follows it is convenient to deal with `generalized subsets' of a
normed algebra.

Let $G$ be an arbitrary set, $M=(a_{\alpha})_{\alpha\in\Lambda}$ and
$N=(b_{\beta})_{\beta\in\Omega}$ families of elements of $G$. We write
$M\subset N$ if there exists an one-to-one map $\varphi$ from $\Lambda$ in
$\Omega$ such that $b_{\varphi(\alpha)}=a_{\alpha}$ for every $\alpha
\in\Lambda$, and $M\simeq N$ if $M\subset N$ and $N\subset M$. The relation
$\simeq$ is clearly an equivalence relation on the set of all families of
elements of $G$. We call the classes of equivalence by \textit{generalized
subsets }of $G$. For brevity we sometimes do not differ a generalized subset
and its arbitrary representative.

It is useful to interpret generalized subsets of $G$ as follows. If $M$ is a
generalized subset of $G$ with a representative $(a_{\alpha})_{\alpha
\in\Lambda}$ then one can clearly identify $M$ with the set $M^{\sharp}$ of
all pairs $(a,t)$, where $a\in A$ and $t=\operatorname*{card}\{\alpha
\in\Lambda:a_{\alpha}=a\}>0$. Under consideration of all such pairs $(a,t)$
for $t\geqslant0$, we also identify $M$ with the cardinal-valued function
$\varkappa_{M}:a\longmapsto t$ defined on $G$. We call $\varkappa_{M}$ the
\textit{functional representation }of $M$. It is clear that one can regard
usual subsets as generalized ones: their functional representations coincide
with their indicators. This justifies the term `generalized subset'.

In terms of the functional representations the inclusion $M\subset N$ for
generalized subsets of $G$ turns into the inequality
\[
\varkappa_{M}(a)\leqslant\varkappa_{N}(a)
\]
for every $a\in G$. It is convenient to define the \textit{union} $M\cup N$
and the \textit{intersection} $M\cap N$ via their functional representations
as follows:
\[
\varkappa_{M\cup N}(a)=\max\{\varkappa_{M}(a),\varkappa_{N}(a)\}
\]
and
\[
\varkappa_{M\cap N}(a)=\min\{\varkappa_{M}(a),\varkappa_{N}(a)\}
\]
for every $a\in G$. Union and intersection of a collection of generalized
subsets are defined similarly.

Now let $A$ be a normed algebra, and let $M,N$ be generalized subsets of $A$.
Let $(a_{\alpha})_{\alpha\in\Lambda}$ and $(b_{\beta})_{\beta\in\Omega}$ be
representatives of $M$ and $N$, respectively. Then we define $MN$ as a
generalized subset of $A$ with the representative $(a_{\alpha}b_{\beta
})_{(\alpha,\beta)\in\Lambda\times\Omega}$. Note that
\[
\varkappa_{MN}(a)=\sum_{(b,c)\in A\times A,\,bc=a}\varkappa_{M}(b)\varkappa
_{N}(c)
\]
for every $a\in A$.

Given a generalized subset $M$ of $A$, set
\begin{equation}
\Vert M\Vert_{1}=\sum_{\alpha\in\Lambda}\Vert a_{\alpha}\Vert\label{f60}%
\end{equation}
(the sum is calculated as $\sup_{\Delta\subset\Lambda}\sum_{\alpha\in\Delta
}\Vert a_{\alpha}\Vert$, where $\Delta$ runs over all finite subsets of
$\Lambda$). Clearly $\Vert M\Vert_{1}$ does not depend on the choice of a
representative of $M$. If $\Vert M\Vert_{1}<\infty$, we say that $M$ is
\textit{summable}. Note that if $M$ is summable then the set $M^{\sharp}$ is
(finite or) countable, $\varkappa_{M}(a)<\infty$ for every $a\in A$ and
\begin{equation}
\Vert M\Vert_{1}=\sum_{a\in A}\varkappa_{M}(a)\Vert a\Vert. \label{f59}%
\end{equation}

Furthermore, clearly
\[
\Vert MN\Vert_{1}\leqslant\Vert M\Vert_{1}\Vert N\Vert_{1},
\]
whence, setting $M^{n}=MM\cdots M$ ($n$ times),
\begin{equation}
\Vert M^{n+m}\Vert_{1}\leqslant\Vert M^{n}\Vert_{1}\Vert M^{m}\Vert_{1}
\label{f58}%
\end{equation}
for every $n,m\in\mathbb{N}$. It follows from (\ref{f58}) that, for every
summable generalized subset $M$ of $A$, there exists a limit
\begin{equation}
\rho_{1}(M)=\lim(\Vert M^{n}\Vert_{1})^{1/n}=\inf(\Vert M^{n}\Vert_{1})^{1/n}.
\label{f61}%
\end{equation}
Note that $(M^{m})^{n}=M^{mn}$ for every $n,m\in\mathbb{N}$, whence
\begin{equation}
\rho_{1}(M^{m})^{1/m}=(\lim_{n}(\Vert(M^{m})^{n}\Vert_{1})^{1/n})^{1/m}%
=\lim_{n}(\Vert M^{mn}\Vert_{1})^{1/nm}=\rho_{1}(M). \label{f62}%
\end{equation}

If $f$ is a map from $A$ into a normed algebra $B$, $fM$ denotes the
generalized subset of $B$ with the representative $(fa_{\alpha})_{\alpha
\in\Lambda}$; in particular, if $I$ is a closed ideal of $A$ then $M/I$
denotes the generalized subset of $A/I$ with the representative $(a_{\alpha
}+I)_{\alpha\in\Lambda}$. Note that
\[
\varkappa_{fM}(b)=\sum_{a\in A,\,fa=b}\varkappa_{M}(a)
\]
for every $b\in B$. If $f$ is a bounded homomorphism then
\begin{equation}
\rho_{1}(fM)=\lim(\Vert(fM)^{n}\Vert_{1})^{1/n}\leqslant\lim\left\|
f\right\|  ^{1/n}(\Vert M^{n}\Vert_{1})^{1/n}\leqslant\rho_{1}(M). \label{f63}%
\end{equation}

Since we consider here usual subsets of $A$ as generalized ones, we calculate
$\rho_{1}$ for them (if they are summable). Of course in these calculations a
power of a subset must be understood in the above sense (that is we multiply
subsets as generalized subsets).

A normed algebra $A$ is called $1$\textit{-quasinilpotent} if $\rho_{1}(M)=0$
for every summable subset $M$ of $A$. Multiplying elements of a summable
generalized subset $M$ of $A$ by different numbers of modulus $1$, it is easy
to point out a summable subset $N$ of $A$ such that $\Vert M^{n}\Vert
_{1}=\Vert N^{n}\Vert_{1}$ for every $n\in\mathbb{N}$, in particular $\rho
_{1}(M)=\rho_{1}(N)$. So, if $A$ is $1$-quasinilpotent then $\rho_{1}(M)=0$
for every summable \textit{generalized} subset $M$ of $A$.

\begin{theorem}
\label{perrad} For a Banach algebra $A$ the following conditions are equivalent.

\begin{itemize}
\item [(i)]$A$ is tensor radical.

\item[(ii)] $A$ is $1$-quasinilpotent.

\item[(iii)] $A{\widehat{\otimes}}B$ is $1$-quasinilpotent, for any Banach
algebra $B$.
\end{itemize}
\end{theorem}

\begin{proof}
The implication (iii)$\Rightarrow$(i) is evident.

(ii)$\Rightarrow$(i) Let $T=\sum a_{k}{\otimes}b_{k}\in A{\widehat{\otimes}}%
B$. We may suppose that $\Vert b_{k}\Vert=1$ and $\sum\Vert a_{k}\Vert<\infty
$. Thus the generalized subset $M$ with the representative $\{a_{k}:k=1,...\}$
is summable. By our assumption $\rho_{1}(M)=0$. Now we have
\begin{equation}
\Vert T^{n}\Vert=\Vert\sum a_{k_{1}}...a_{k_{n}}{\otimes}b_{k_{1}}...b_{k_{n}%
}\Vert\leqslant\sum\Vert a_{k_{1}}...a_{k_{n}}\Vert=\Vert M^{n}\Vert_{1}
\label{f50}%
\end{equation}
and accordingly $\rho(T)\leqslant\rho_{1}(M)=0$.

(i)$\Rightarrow$(ii) Let $G=\mathcal{S}_{1}(W)$ be the free unital semigroup
with a countable set $W=\{w_{k}\}_{k\geqslant1}$ of generators. That is
$G=\cup_{m\geqslant0}W_{m}$, where $W_{0}={1}$, the direct product
\[
W_{m}=W\times W...\times W
\]
is the set of `words' $w_{k_{1}}w_{k_{2}}...w_{k_{m}}$ of the length $m$, and
the multiplication is lexical. Let $B=l^{1}(G)$ be the corresponding semigroup
algebra. We show that if $A{\widehat{\otimes}}B$ is radical then $A$ is $1$-quasinilpotent.

Indeed, for any summable (generalized) subset $M$ of $A$ with a representative
$\{a_{k}:k=1,2,\ldots\}$, we define the element $T_{M}\in A{\widehat{\otimes}%
}B$ by
\begin{equation}
T_{M}=\sum a_{k}{\otimes}w_{k}. \label{fT}%
\end{equation}
Then
\[
T_{M}^{n}=\sum a_{k_{1}}...a_{k_{n}}{\otimes}w_{k_{1}}...w_{k_{n}}.
\]
Since $A{\widehat{\otimes}}l^{1}(G)$ is isometrically isomorphic via the map
$\pi:(a{\otimes}f)(g)\rightarrow f(g)a$ to the Banach algebra $l^{1}(G,A)$ of
all summable $A$-valued functions on $G$,
\begin{equation}
\Vert T_{M}^{n}\Vert=\sum\Vert a_{k_{1}}...a_{k_{n}}\Vert=\Vert M^{n}\Vert
_{1}.
\end{equation}
It follows that
\begin{equation}
\rho(T_{M})=\rho_{1}(M) \label{f3.6}%
\end{equation}
and $\rho_{1}(M)=0$.

(i)$\Rightarrow$(iii) Since $A$ is tensor radical, the algebra $(A{\widehat
{\otimes}}B){\widehat{\otimes}}C=A{\widehat{\otimes}}(B{\widehat{\otimes}}C)$
is radical, for any $B$ and $C$. Hence $A{\widehat{\otimes}}B$ is tensor
radical. Applying to $A{\widehat{\otimes}}B$ the implication (i)$\Rightarrow
$(ii), we see that $A{\widehat{\otimes}}B$ is $1$-quasinilpotent.
\end{proof}

\begin{corollary}
\label{4.13} A Banach algebra $A$ is tensor radical ($=1$-quasinilpotent) if
and only if $A{\widehat{\otimes}}l^{1}(G)$ is radical, where $G$ is a free
unital semigroup with countable set of generators.
\end{corollary}

\begin{proof}
``Only if'' is evident, ``if'' follows from (\ref{f3.6}).
\end{proof}

Thus the problem of $1$-quasinilpotence of all radical Banach algebras reduces
to the problem of tensor perfectness for a single algebra, namely for
$l^{1}(G)$.

Now we characterize $\operatorname{Rad}^{t}(A)$ in terms of $\rho_{1}$.

\begin{theorem}
\label{1quas}Let $A\in(BA)$. For $a\in A$ the following conditions are equivalent.

\begin{itemize}
\item [(i)]$a\in\operatorname{Rad}^{t}(A)$.

\item[(ii)] $\rho_{1}(aM)=0$, for any summable subset $M$ of $A$.

\item[(iii)] $a{\widehat{\otimes}}1\in\operatorname{Rad}(A{\widehat{\otimes}%
}l^{1}(G))$, where $G=\mathcal{S}_{1}(W)$ is the free unital semigroup with
countable set $W$ of generators.
\end{itemize}
\end{theorem}

\begin{proof}
(ii)$\Rightarrow$(i) Let $B\in(BA)$, and let $T=\sum a_{i}{\otimes}b_{i}\in
A{\widehat{\otimes}}B$ with $\Vert b_{i}\Vert=1$ and $\sum\Vert a_{i}%
\Vert<\infty$. Then
\[
(a{\otimes}b)T=\sum aa_{i}{\otimes}bb_{i}%
\]
for arbitrary $b\in B$. Multiplying by numbers of modulus $1$, one can assume
that $a_{i}\neq a_{j}$ if $i\neq j$; Let $M$ be a subset of $A$ consisting of
all $a_{i}$. Then $aM$ has the representative $\{aa_{i}:i\geqslant1\}$ and
\[
\rho((a{\otimes}b)T)\leqslant\Vert b\Vert\rho_{1}(aM)
\]
(the estimate is similar to (\ref{f50}) in the proof of Theorem \ref{perrad}).
Since $\rho_{1}(aM)=0$, $(a{\otimes}b)T$ is quasinilpotent. Since this holds
for arbitrary $T\in A{\widehat{\otimes}}B$, $a{\otimes}b\in\operatorname{Rad}%
(A{\widehat{\otimes}}B)$. Therefore $a\in\operatorname{Rad}^{t}(A)$.

(i)$\Rightarrow$(iii) is evident.

(iii)$\Rightarrow$(ii) Let $M$ be a summable subset of $A$ with a
representative $\{a_{k}:k\geqslant1\}$. Then the element $T_{aM}=\sum
(aa_{k}{\otimes}w_{k})$ of $A{\widehat{\otimes}}l^{1}(G)$ is quasinilpotent
because $T_{aM}=(a{\otimes}1)T_{M}$ (see the definition of $T_{M}$ in
(\ref{fT})). By~ (\ref{f3.6}), $\rho_{1}(aM)=0$ and this shows that $a$
satisfies (ii).
\end{proof}

\begin{theorem}
\label{4.16.1} $\operatorname{Rad}^{t}$ is a symmetric uniform tensor HTR on
$(BA)$.
\end{theorem}

\begin{proof}
It follows from Theorem \ref{rtrad} that to see that $\operatorname{Rad}^{t}$
is an HTR we need only to prove the inclusion $I\cap\operatorname{Rad}%
^{t}(A)\subset\operatorname{Rad}^{t}(I)$. This easily follows from Theorem
\ref{1quas}. Indeed, if $a\in I\cap\operatorname{Rad}^{t}(A)$ then $\rho
_{1}(aM)=0$, for every summable subset of $I$, because it is true for subsets
of $A$.

The fact that $\operatorname{Rad}^{t}$ is a tensor radical whenever it is a
radical was mentioned before.

Since a (closed) subalgebra of a $1$-quasinilpotent algebra is clearly
$1$-quasinilpotent, $\operatorname{Rad}^{t}$ is uniform. It is clear that
$\operatorname{Rad}^{t}$ is symmetric.
\end{proof}

\begin{corollary}
\label{4.16} $\operatorname{Rad}^{t}(A)$ is the largest tensor radical ideal
of a Banach algebra $A$. Moreover, it is closed and contains all (one-sided,
non-necessarily closed) $1$-quasinilpotent ideals of $A$.
\end{corollary}

\begin{proof}
The first part is a consequence of Theorem \ref{4.16.1} and general properties
of a TR (Lemma \ref{2.1}). The last statement follows immediately from the
part (ii) of Theorem \ref{1quas}.
\end{proof}

Using the extension property of topological radicals (Theorem \ref{2.3}), one obtains

\begin{corollary}
\label{trext} Let $A\in(BA)$, and let $I$ be its closed ideal. If $A/I$ and
$I$ are $1$-quasinilpotent then so is $A$.
\end{corollary}

Now we will show that $\rho_{1}(M)$ can be calculated in the algebra
$A/\operatorname{Rad}^{t}(A)$ for every summable family $M$ of elements of $A$.

\begin{theorem}
\label{3quot} For any summable generalized subset $M$ of elements of a Banach
algebra $A$,
\[
\rho_{1}(M)=\rho_{1}(M/\operatorname{Rad}^{t}(A)).
\]
\end{theorem}

\begin{proof}
Let $J=\operatorname{Rad}^{t}(A)$ and $B=l^{1}(G)$, where $G$ is the free
semigroup with countable set of generators. The ideal $i_{A}(J{\widehat
{\otimes}}B)$ of $A{\widehat{\otimes}}B$ is contained in $\operatorname{Rad}%
(A{\widehat{\otimes}}B)$ and the same is true for its closure $U$. Let
$E=T_{M}\in A{\widehat{\otimes}}B$, where $T_{M}$ is defined by (\ref{fT}),
and $F=E+U$ the image of $E$ in $(A{\widehat{\otimes}}B)/U$. Then
$\rho(E)=\rho(F)$ since $U\subset\operatorname{rad}(A{\widehat{\otimes}}B)$.
This and (\ref{f3.6}) imply $\rho_{1}(M)=\rho(F)$. Let $\tau_{J}%
:(A/J){\widehat{\otimes}}B\rightarrow(A{\widehat{\otimes}}B)/U$ be the
epimorphism introduced in Lemma \ref{ideal}. Then clearly $F=\tau_{J}%
(T_{M/J})$, whence
\[
\rho_{1}(M)=\rho(F)=\rho(\tau_{J}(T_{M/J}))\leqslant\rho(T_{M/J})=\rho
_{1}(M/J).
\]

We proved that $\rho_{1}(M)\leqslant\rho_{1}(M/J)$. The converse inequality is evident.
\end{proof}

Note that the theorem fails being formulated for usual subsets of $A$ and
their ``usual'' (non-generalized) images in $A/\operatorname{Rad}^{t}(A)$.

\subsection{Tensor radical and tensor perfect algebras}

Our next result shows that the class of tensor perfect algebras is extension stable.

\begin{theorem}
\label{4.3} If a closed ideal $J$ of a Banach algebra $A$ and the quotient
$A/J$ are tensor perfect then $A$ is tensor perfect.
\end{theorem}

\begin{proof}
Let $B$ be a radical Banach algebra and $\pi\in\operatorname{irr}%
(A{\widehat{\otimes}}B)$. We have to prove that $\pi=0$. Since the algebra
$J{\widehat{\otimes}}B$ is radical and $i_{A}(J{\widehat{\otimes}}B)$ is its
homomorphic image, $i_{A}(J{\widehat{\otimes}}B)$ consists of quasinilpotent
elements. Being an ideal of $A{\widehat{\otimes}}B$, it is contained in
$\operatorname{Rad}(A{\widehat{\otimes}}B)$. Therefore the closure
$\widetilde{i_{A}(J{\widehat{\otimes}}B)}\subset\ker\pi$. This allows one to
define a representation $\pi^{\prime}$ of $(A{\widehat{\otimes}}%
B)/\widetilde{i_{A}(J{\widehat{\otimes}}B)}$ with the same range as $\pi$. It
follows from Lemma \ref{ideal} that the representation $\pi^{\prime\prime}%
=\pi^{\prime}\circ\tau_{J}$ of $(A/J){\widehat{\otimes}}B$ also has the same
image as $\pi$. Therefore $\pi^{\prime\prime}\in\operatorname{irr}%
((A/J){\widehat{\otimes}}B)$. Since $(A/J){\widehat{\otimes}}B$ is radical,
$\pi^{\prime\prime}=0$ and $\pi=0$.
\end{proof}

In what follows we list some classes of tensor radical and tensor perfect
algebras. More strong and general results will be obtained in the next section
and further parts of the project.

\begin{lemma}
\label{4.4} Any nilpotent Banach algebra is tensor radical.
\end{lemma}

\begin{proof}
Evidently follows from the fact that the tensor product of a nilpotent algebra
and an arbitrary algebra is nilpotent.
\end{proof}

\begin{proposition}
\label{4.5} Any Banach algebra with dense socle is tensor perfect.
\end{proposition}

\begin{proof}
Let $p$ be a minimal idempotent in $A$ (in the sense that $\dim(pAp)=1$) and
$C=pA$. The ideal $J=C(1-p)$ of $C$ is nilpotent and consequently tensor
perfect. The quotient $C/J$, being isomorphic to $Cp=pAp$, is one-dimensional
and hence tensor perfect. By Theorem \ref{4.3} $C$ is tensor perfect. Let $B$
be a radical Banach algebra. The ideal $i_{A}(C{\widehat{\otimes}}B)$ of
$A{\widehat{\otimes}}B$ consists of quasinilpotent elements (being a
homomorphic image of $C{\widehat{\otimes}}B$). Hence $i_{A}(C{\widehat
{\otimes}}B)\subset\operatorname{Rad}(A{\widehat{\otimes}}B)$. In particular
$pa\otimes b\in\operatorname{Rad}(A{\widehat{\otimes}}B)$ for all $a\in A$,
$b\in B$. Since any element of $\operatorname{soc}(A)$ is a linear combination
of elements of the form $paq$, where $a\in A$, $p$ and $q$ are minimal
idempotents, we have that $a\otimes b\in\operatorname{Rad}(A{\widehat{\otimes
}}B)$, for any $a\in\operatorname{soc}(A)$, $b\in B$. By the density
assumption, $a\otimes b\in\operatorname{Rad}(A{\widehat{\otimes}}B)$, for all
$a\in A$, $b\in B$. This means that $A{\widehat{\otimes}}B$ is radical.
\end{proof}

\begin{lemma}
\label{4.7} A Banach algebra $A$ is tensor perfect iff
\[
i_{B}(A{\widehat{\otimes}}\operatorname{Rad}(B))\subset\operatorname{Rad}%
(A{\widehat{\otimes}}B),
\]
for every Banach algebra $B$.
\end{lemma}

\begin{proof}
Suppose that $A$ is tensor perfect. Then as above, $i_{B}(A{\widehat{\otimes}%
}\operatorname{Rad}(B))$, being a homomorphic image of $A{\widehat{\otimes}%
}\operatorname{Rad}(B)$, consists of quasinilpotent elements. Hence this ideal
is contained in $\operatorname{Rad}(A{\widehat{\otimes}}B)$. The inverse
implication is evident.
\end{proof}

The following result can be deduced from \cite{A79}. Its proof for the case of
unital algebras is similar to one of \cite[Theorem 4.4.2]{A79}; we present the
proof because its `non-unital' part \cite[Corollary 4.4.1]{A79} seems to be
written too briefly.

\begin{lemma}
\label{4.8} If a Banach algebra $A$ is commutative then, for any Banach
algebra $B$,
\begin{equation}
i_{A}(\operatorname{Rad}(A){\widehat{\otimes}}B)\cup i_{B}(A{\widehat{\otimes
}}\operatorname{Rad}(B))\subset\operatorname{Rad}(A{\widehat{\otimes}}B)
\label{f3.1}%
\end{equation}
\end{lemma}

\begin{proof}
Suppose firstly that $A$ and $B$ are unital. For any $\pi\in\operatorname{irr}%
(A{\widehat{\otimes}}B)$, let $\pi_{1}(a)=\pi(a\otimes1)$, $\pi_{2}%
(b)=\pi(1\otimes b)$, then $\pi_{1}(A)$ commutes with $\pi_{1}(A)$ and
$\pi_{2}(B)$. This means that $\pi_{1}(A)$ is contained in the centre of
$\pi(A{\widehat{\otimes}}B)$. Strict irreducibility of $\pi$ implies that the
centre of $\pi(A{\widehat{\otimes}}B)$ is one-dimensional, i.e. in fact
$\pi_{1}(a)=h(a)$ for a character $h$ of $A$. Hence%
\begin{equation}
\pi(\sum a_{n}{\otimes}b_{n})=\sum h(a_{n})\pi_{2}(b_{n})=\pi_{2}(\sum
h(a_{n})b_{n}). \label{f3.2}%
\end{equation}
Therefore $\pi(A{\widehat{\otimes}}B)=\pi_{2}(B)$ and $\pi_{2}\in
\operatorname{irr}B$. It follows that $\pi_{2}(\operatorname{Rad}(B))=0$ and,
by (\ref{f3.2}),
\[
\pi(i_{B}(A{\widehat{\otimes}}\operatorname{Rad}(B)))=0.
\]
Similarly, since $h(\operatorname{Rad}(A))=0$, we have that
\[
\pi(i_{A}(\operatorname{Rad}(A){\widehat{\otimes}}B))=0.
\]
Since $\pi$ is arbitrary, the equality (\ref{f3.1}) is proved.

In the general case, the ideals $J_{1}=i_{A^{1}}(\operatorname{Rad}%
(A^{1}){\widehat{\otimes}}B^{1})$ and $J_{2}=i_{B^{1}}(A^{1}{\widehat{\otimes
}}\operatorname{Rad}(B^{1}))$ are contained in $\operatorname{Rad}%
(A^{1}{\widehat{\otimes}}B^{1})$, whence they are topologically nil. Since $A$
and $B$ are complemented subspaces in $A^{1}$ and $B^{1}$ respectively, the
algebra $A{\widehat{\otimes}}B$ is topologically isomorphic to its canonical
image $J$ in $A^{1}{\widehat{\otimes}}B^{1}$ (in particular $J$ is an ideal of
$A^{1}{\widehat{\otimes}}B^{1}$). Hence $J_{1}\cap J$ and $J_{2}\cap J$ are
topologically nil ideals of $J$ and this clearly implies our assertion.
\end{proof}

The following statement is an obvious consequence of our results and Lemma
\ref{4.8}.

\begin{corollary}
\label{4.10} Any commutative Banach algebra is tensor perfect, and any radical
commutative Banach algebra is tensor radical.
\end{corollary}

\section{\label{S4} TRs related to the joint quasinilpotence}

Let $A$ be a normed algebra. If $N\subset A$, let $\widetilde{N}$ denote the
closure of $N$ in $A$ and $\operatorname{abs}(N)$ the closure of absolutely
convex hull of $N$. We define here the (usual) sum and product of subsets $M$
and $N$ of $A$ by
\[
M+N=\{a+b:a\in M,b\in N\}
\]
and
\begin{equation}
MN=\{ab:a\in M,b\in N\}, \label{f71}%
\end{equation}
respectively; it does similarly for several subsets and for a power $M^{n}$,
$n\in\mathbb{N}$. The set $\mathcal{S}(M)=$ $\cup_{n\geqslant1}M^{n}$ is
called the \textit{semigroup generated by} $M$. If $A$ is unital,
$\mathcal{S}_{1}(M)=\mathcal{S}(M)\cup\{1\}$.

In what follows we denote by ${\mathcal{M}}_{f}(A)$, ${\mathcal{M}}_{c}(A)$
and ${\mathcal{M}}_{b}(A)$ the classes of all finite, precompact and bounded
subsets of a normed algebra $A$, respectively. For $M\in{\mathcal{M}}_{b}(A)$,
one sets
\[
\left\|  M\right\|  =\sup_{a\in M}\left\|  a\right\|
\]
(the norm of $M$) and
\[
\rho(M)=\inf\left\|  M^{n}\right\|  ^{1/n}.
\]
Since the norm is submultiplicative,%
\begin{equation}
\rho(M)=\inf\left\|  M^{n}\right\|  ^{1/n}=\lim\left\|  M^{n}\right\|  ^{1/n}.
\label{f72}%
\end{equation}
The number $\rho(M)$ is called \cite{RS60} the \textit{joint spectral radius}
of $M$. It is clear that
\[
\rho(M)=\rho(M^{n})^{1/n}%
\]
for every $n\in\mathbb{N}$,
\[
\rho(\lambda M)=\left|  \lambda\right|  \rho(M)
\]
for every $\lambda\in\mathbb{C}$, and
\begin{equation}
\rho(MN)=\rho(NM) \label{f74}%
\end{equation}
for every $N\in\mathcal{M}_{b}(A).$

Given here the definition of product of subsets of a normed algebra differs
from one for generalized subsets defined earlier (in Subsection \ref{sub}).
However the calculating $\rho(M)$ by rules of generalized subsets gives the
same result as in (\ref{f72}) because, for a generalized subset $M$ of $A$,
the norm $\left\|  M\right\|  =\sup\{\left\|  a\right\|  :(a,t)\in M^{\sharp
}\}$ does not depend on cardinal numbers and coincides with the norm of the
usual set $\{a:(a,t)\in M^{\sharp}\}$. So we prefer to use (\ref{f71}) in what
follows as far as possible.

We need the following known properties of $\rho$.

\begin{lemma}
\label{st1}Let $A$ be a normed algebra and $N,M\in\mathcal{M}_{b}(A)$.

\begin{itemize}
\item [(i)]Then $\rho(L)\leqslant\rho(M)=\rho(\operatorname{abs}(M))$ for any
$L\subset\operatorname{abs}(M)$. In particular $\rho(M)=\rho(\widetilde{M})$.

\item[(ii)] If every element of $N$ commutes with every element of $M$ then
\begin{align*}
\rho(M\cup N)  &  =\max\{\rho(M),\rho(N)\},\\
\rho(M+N)  &  \leqslant\rho(M)+\rho(N),\quad\rho(MN)\leqslant\rho(M)\rho(N).
\end{align*}
\end{itemize}
\end{lemma}

\begin{proof}
(i) The equality $\rho(M)=\rho(\operatorname{abs}(M))$ was proved in
\cite[Proposition 2.6]{ST00}, the other properties are evident.

(ii) was proved in \cite[Lemma 2.13]{ST00}.
\end{proof}

\begin{theorem}
[{\cite[Theorem 3.5]{ST00}}]\label{st2}Let $A$ be a normed algebra, and
$V\subset\mathbb{C}$ be open. Let $F$ be a family of analytic functions from
$V$ into $A$ such that
\[
\lim_{\mu\rightarrow\lambda}\sup\left\{  \left\|  f(\mu)-f(\lambda)\right\|
:f\in F\right\}  =0
\]
for every $\lambda\in V$ and all $M(\lambda)=\{f(\lambda):f\in F\}\in
\mathcal{M}_{b}(A)$. Then $\lambda\longmapsto\log\rho(M(\lambda))$ and
$\lambda\longmapsto\rho(M(\lambda))$ are subharmonic on $V$.
\end{theorem}

For subharmonicity, see \cite{HK76}.

An important general problem of the joint spectral radius theory is to express
$\rho$ for precompact sets of operators in $B(X)$ via their spectral
characteristics; see some partial results in \cite{T99, ST00, ST02}.

\subsection{Joint quasinilpotent algebras and ideals}

A normed algebra $A$ is called \textit{finitely quasinilpotent} (resp.
\textit{compactly quasinilpotent}; \textit{boundedly quasinilpotent}) if
$\rho(M)=0$, for every $M\in{\mathcal{M}}_{f}(A)$ (resp. ${\mathcal{M}}%
_{c}(A)$; ${\mathcal{M}}_{b}(A)$). For brevity, we write that $A$ is
$f$-quasinilpotent (resp. $c$-quasinilpotent; $b$-quasinilpotent).

\begin{lemma}
\label{3.0} Let $A$ be a normed algebra.

\begin{itemize}
\item [(i)]If $A$ is $c$-quasinilpotent or $b$-quasinilpotent then the same is
true for its completion $\overline{A}$.

\item[(ii)] Let $\phi:A\rightarrow B$ be a continuous homomorphism. If $A$ is
$f$-quasinilpotent then ${\phi}(A)$ is $f$-quasinilpotent.
\end{itemize}
\end{lemma}

\begin{proof}
(i) follows from the fact that any bounded (precompact) subset of
$\overline{A}$ is contained in the closure of a bounded (respectively
precompact) subset of $A$.

To see (ii) it suffices to note that any $M\in{\mathcal{M}}_{f}({\phi}(A))$ is
of the form $\phi(N)$, for some $N\in{\mathcal{M}}_{f}(A)$, whence $\Vert
M^{n}\Vert\leqslant\Vert\phi\Vert\Vert N^{n}\Vert$ for every $n\in\mathbb{N}$
and $\rho(M)=0$.
\end{proof}

Boundedly quasinilpotent algebras are usually called \textit{topologically
nilpotent} \cite{D91, DM92, DW92}. Part (i) of the previous lemma, for this
case, is mentioned, for example, in \cite{PW99}. For $\ast\in\{f,c,b\}$, an
ideal of a normed algebra is called a $\ast$\textit{-quasinilpotent ideal }if
it is a $\ast$-quasinilpotent algebra.

\subsection{Ideals $R_{f}$, $R_{c}$ and $R_{b}$}

We set, for $\ast\in\{f,c,b\}$,
\[
R_{\ast}(A)=\{a\in A:\rho(\{a\}\cup M)=\rho(M),\quad\forall M\in{\mathcal{M}%
}_{\ast}(A)\}.
\]
Clearly the condition $\rho(\{a\}\cup M)=\rho(M)$ can be rewritten in the
form
\[
\rho(\{a\}\cup M)\leqslant\rho(M)
\]
because the reverse inequality is evident.

\begin{lemma}
\label{3.1} If $a\in R_{\ast}(A)$ then $\lambda a\in R_{\ast}(A)$, for any
$\lambda\in\mathbb{C}$.
\end{lemma}

\begin{proof}
Indeed,
\begin{align*}
\rho(\{\lambda a\}\cup M)  &  =\rho(\lambda(\{a\}\cup\lambda^{-1}%
M))=|\lambda|\rho(\{a\}\cup\lambda^{-1}M)\\
&  =|\lambda|\rho(\lambda^{-1}M)=\rho(M).
\end{align*}
\end{proof}

\begin{lemma}
\label{3.2} $a\in R_{\ast}(A)$ iff $\sup\{\rho(\{\lambda a\}\cup M):\lambda
\in\mathbb{C\}}<\infty$ for every $M\in\mathcal{M}_{\ast}(A)$.
\end{lemma}

\begin{proof}
The function $f(\lambda)=\rho(\{\lambda a\}\cup M)$ is subharmonic by Theorem
\ref{st2} and bounded on $\mathbb{C}$; by \cite{HK76}, it is constant. Hence
\[
\rho(\{a\}\cup M)=f(1)=f(0)=\rho(M).
\]
\end{proof}

Our next lemma shows that one may assume $A$ to be unital.

\begin{lemma}
\label{3.3} $R_{\ast}(A)=R_{\ast}(A^{1})$.
\end{lemma}

\begin{proof}
Any set $M\in\mathcal{M}_{\ast}(A^{1})$ is contained in a set of the form
$N+L$, where $N\in\mathcal{M}_{\ast}(A)$ and $L\in\mathcal{M}_{\ast
}(\mathbb{C})$. Therefore it suffices to show that the function $\lambda
\longmapsto\rho(\{\lambda a\}\cup(N+L))$ is bounded, for any $a\in R_{\ast
}(A)$. Without loss of generality, we may assume that $0\in L$. Hence
\[
\{\lambda a\}\cup(N+L)\subset(\{\lambda a\}\cup N)+L
\]
and, by Lemma \ref{st1},
\begin{align*}
\rho(\{\lambda a\}\cup(N+L))  &  \leqslant\rho((\{\lambda a\}\cup N)+L)\\
&  \leqslant\rho(\{\lambda a\}\cup N)+\left\|  L\right\|  =\rho(N)+\left\|
L\right\|  .
\end{align*}
\end{proof}

By induction we have

\begin{lemma}
\label{3.4} If $N$ is a finite subset of $R_{\ast}(A)$ then $\rho(N\cup
M)=\rho(M)$, for any $M\in\mathcal{M}_{\ast}(A)$.
\end{lemma}

\begin{lemma}
\label{3.5} $R_{*}(A)$ is a linear subspace of $A$.
\end{lemma}

\begin{proof}
For $a,b\in R_{\ast}(A)$, set $N=\{a,b\}$. Since $\{(a+b)/2\}\cup
M\subset\operatorname{abs}(N\cup M)$, we obtain from Lemmas \ref{st1} and
\ref{3.4} that
\[
\rho(\{(a+b)/2\}\cup M)\leqslant\rho(\operatorname{abs}(N\cup M))=\rho(N\cup
M)=\rho(M),
\]
for any $M\in\mathcal{M}_{\ast}(A)$. This shows that $(a+b)/2\in R_{\ast}(A)$
and it remains to use Lemma \ref{3.1}.
\end{proof}

\begin{lemma}
\label{3.6} $R_{*}(A)$ is an ideal of $A$.
\end{lemma}

\begin{proof}
By Lemma \ref{3.3}, one may assume that $A$ is unital. Let $a\in R_{\ast}(A)$,
$b\in A$, $M\in\mathcal{M}_{\ast}(A)$, $\lambda\in\mathbb{C}$. Setting
$N=\{\lambda a,b,1\}$, we get:
\[
\rho(N\cup M)\leqslant\beta=\rho(\{b,1\}\cup M).
\]
It follows that $\rho((N\cup M)^{2})\leqslant\beta^{2}$. But $(N\cup M)^{2}$
contains $\{\lambda ab\}\cup M$. Hence
\[
\sup_{\lambda\in\mathbb{C}}\rho(\{\lambda ab\}\cup M)\leqslant\beta^{2}%
<\infty,
\]
whence $ab\in R_{\ast}(A)$ by Lemma \ref{3.2}. Similarly $ba\in R_{\ast}(A)$.
\end{proof}

\begin{theorem}
\label{3.7} $R_{*}(A)$ is a closed ideal of $A$.
\end{theorem}

\begin{proof}
Note that if $a\in R_{\ast}(A)$, $M\in\mathcal{M}_{\ast}(A)$ then, for any
$b\in A$,
\begin{align*}
\rho(\{(a+b)\}\cup M)  &  \leqslant\rho(\operatorname{abs}(\{2a,2b\}\cup M))\\
&  =\rho(\{2a,2b\}\cup M))=\rho(\{2b\}\cup M)\\
&  \leqslant\Vert\{2b\}\cup M\Vert=\max\{2\Vert b\Vert,\Vert M\Vert\}.
\end{align*}
Now if $c\in\widetilde{R_{\ast}(A)}$ then there are $a\in R_{\ast}(A)$, $b\in
A$ with $c=a+b$, $\Vert b\Vert<\Vert M\Vert/2$. Therefore
\[
\rho(\{c\}\cup M)\leqslant\max\{2\Vert b\Vert,\Vert M\Vert\}=\Vert M\Vert.
\]
Changing $c$ by $\lambda c$ and applying Lemma \ref{3.2}, we conclude that
$c\in R_{\ast}(A)$ and $R_{\ast}(A)$ is closed.
\end{proof}

\subsection{Properties of ideals $R_{f}$, $R_{c}$ and $R_{b}$}

Clearly all $R_{\ast}(A)$ are topologically nil ideals of a normed algebra
$A$. Hence, for any normed algebra $A$, $R_{\ast}(A)\subset\operatorname{rad}%
_{n}(A)$ by Theorem \ref{1.7}(ii) or, more transparently,
\begin{equation}
R_{b}(A)\subset R_{c}(A)\subset R_{f}(A)\subset\operatorname{rad}_{n}(A).
\label{fb}%
\end{equation}
Moreover
\[
R_{c}(A)\subset\operatorname{Rad}^{\prime}(A).
\]
This is a consequence of the following `regularity' property:

\begin{lemma}
\label{reg} Let $A$ be a normed algebra. Then

\begin{itemize}
\item [(i)]$R_{b}(A)=A\cap R_{b}(\overline{A})$.

\item[(ii)] $R_{c}(A)=A\cap R_{c}(\overline{A})$.
\end{itemize}
\end{lemma}

\begin{proof}
The statements follow from Lemma \ref{3.0} and the fact that any bounded
(resp. precompact) subset of $\overline{A}$ is contained in the closure of a
bounded (resp. precompact) subset of $A$.
\end{proof}

\begin{lemma}
\label{3.20} Let $\ast\in\{c,b\}$, and let $A$ be a normed algebra. If $N$ is
a precompact subset of $R_{\ast}(A)$ then $\rho(N\cup M)=\rho(M)$, for any
$M\in\mathcal{M}_{\ast}(A)$.
\end{lemma}

\begin{proof}
Let $G(\lambda)=\lambda N\cup M$, for $\lambda\in\mathbb{C}$. By \cite[Lemma
8.2.2]{RR64}, there is a sequence $(a_{n})$ in $R_{c}(A)$ convergent to $0$
with $N\subset\operatorname{abs}\{a_{n}:n\in\mathbb{N}\}$. Let $F_{m}%
=\{a_{n}:n\geqslant m\}$, for $m\in\mathbb{N}$. Then $\left\|  F_{m}\right\|
\rightarrow0$ as $m\rightarrow\infty$. By Lemmas \ref{st1} and \ref{3.4},
\begin{align*}
\rho(G(\lambda))  &  =\rho(\lambda N\cup M)\leqslant\rho(\lambda
(\operatorname{abs}F_{1})\cup M)\leqslant\rho(\operatorname{abs}(\lambda
F_{1}\cup M))\\
&  =\rho(\lambda F_{1}\cup M)=\rho(\{\lambda a_{1},\ldots,\lambda
a_{m-1}\}\cup\lambda F_{m}\cup M)\\
&  =\rho(\lambda F_{m}\cup M)\leqslant\left\|  \lambda F_{m}\cup M\right\|
=\max\{\left|  \lambda\right|  \left\|  F_{m}\right\|  ,\left\|  M\right\|  \}
\end{align*}
for every $m\in\mathbb{N}$. Take $m$ such that $\left\|  F_{m}\right\|
<\left|  \lambda\right|  ^{-1}$. Then
\[
\rho(G(\lambda))\leqslant\max\{1,\left\|  M\right\|  \},
\]
for every $\lambda\in\mathbb{C}$. Being a subharmonic function on $\mathbb{C}$
by Theorem \ref{st2}, $\lambda\longmapsto\rho(G(\lambda))$ is constant. Hence
\[
\rho(M)=\rho(G(0))=\rho(G(1))=\rho(N\cup M).
\]
\end{proof}

\begin{proposition}
\label{3.8}Let $A$ be a normed algebra. Then

\begin{itemize}
\item [(i)]$R_{f}(A)$ is $f$-quasinilpotent.

\item[(ii)] $R_{c}(A)$ is $c$-quasinilpotent.
\end{itemize}
\end{proposition}

\begin{proof}
The statements follow trivially from Lemmas \ref{3.4} and \ref{3.20}.
\end{proof}

The ideal $R_{b}(A)$ is not necessarily $b$-quasinilpotent. To show this, let
us look at the case that $A$ is commutative.

\begin{proposition}
\label{3.22}Let $A$ be a commutative normed $Q$-algebra. Then
$\operatorname{rad}(A)=R_{b}(A).$
\end{proposition}

\begin{proof}
By Theorem \ref{p1} and (\ref{fb}), it suffices to show that
$\operatorname{rad}(A)\subset R_{b}(A)$. If $a\in\operatorname{rad}(A)$ and
$M\in\mathcal{M}_{b}(A)$ then
\[
\rho(\{\lambda a\}\cup M)=\max\{\rho(\lambda a),\rho(M)\}=\rho(M)
\]
by Lemma \ref{st1}, for every $\lambda\in\mathbb{C}$, whence $a\in R_{b}(A)$
by Lemma \ref{3.2}.
\end{proof}

It is known that there exist commutative radical Banach algebras which are not
$b$-quasinilpotent, for example the norm-closed algebra generated by the
Volterra integration operator (see \cite{PW99}, where many other examples can
be found). On the other hand, each commutative radical Banach algebra $A$ is
$c$-quasinilpotent because it follows from the above proposition that
$R_{c}(A)=\operatorname*{Rad}(A)$.

Recall that $\mathcal{S}_{1}(M)$ denotes the unital semigroup generated by a
subset $M$ of a normed algebra.

\begin{theorem}
\label{3.12}Let $A$ be a unital normed algebra and $N,M\subset\mathcal{M}%
_{b}(A)$. If $N\mathcal{S}_{1}(M)$ is bounded and $\rho(N\mathcal{S}%
_{1}(M))=0$ then $\rho(N\cup M)=\rho(M)$.
\end{theorem}

\begin{proof}
Let $G=\mathcal{S}_{1}(M)$ and $\gamma=\max\{\Vert M\Vert,1\}$. By the
assumption, for any $\varepsilon>0$ there is $\beta=\beta(\varepsilon)$ such
that
\[
\Vert(NG)^{n}\Vert\leqslant\beta\varepsilon^{n},
\]
for all $n$. Given a number $\lambda\in\mathbb{C}$, set $E_{\lambda}={\lambda
N}\cup M$. We choose $\varepsilon<|{\lambda}|^{-1}$. Let $x=x_{1}x_{2}%
...x_{n}\in(E_{\lambda})^{n}$ and let exactly $k$ of the elements $x_{j}$
belong to $\lambda N$. Then
\[
x=\lambda^{k}zy_{1}y_{2}...y_{k},
\]
where all $y_{i}\in NG$, $z\in G$ and $\left\|  z\right\|  \leqslant\gamma
^{n}$. It follows that
\[
\left\|  x\right\|  \leqslant|\lambda|^{k}\beta\varepsilon^{k}\gamma
^{n}\leqslant\beta\gamma^{n}.
\]
This shows that $\left\|  (E_{\lambda})^{n}\right\|  \leqslant\beta\gamma^{n}$
for all $n$, whence
\[
\rho(E_{\lambda})\leqslant\gamma
\]
for every $\lambda\in\mathbb{C}$. Being a subharmonic function on $\mathbb{C}%
$, $\lambda\longmapsto$ $\rho(E_{\lambda})$ is constant, whence $\rho
(E_{1})=\rho(E_{0})$ which is what we need.
\end{proof}

\begin{lemma}
\label{3.21}Let $\ast\in\{c,b\}$, and let $J$ be an (one-sided) $\ast
$-quasinilpotent ideal of a normed algebra $A$. If $N$ $\in\mathcal{M}_{\ast
}(J)$ then $\rho(N\cup M)=\rho(M)$, for any $M\in\mathcal{M}_{\ast}(A)$.
\end{lemma}

\begin{proof}
Let $J$ be a right ideal of $A$. One may suppose that $A$ is unital. Suppose
firstly that $\left\|  M\right\|  <1$. Then $N\mathcal{S}_{1}(M)\in
\mathcal{M}_{\ast}(J)$, so that $\rho(N\mathcal{S}_{1}(M))=0$. By Theorem
\ref{3.12},
\[
\rho(N\cup M)=\rho(M).
\]
If $\left\|  M\right\|  \geqslant1$, take $t>\left\|  M\right\|  $. Then
\[
\rho(N\cup M)=t\rho(t^{-1}N\cup t^{-1}M)=t\rho(t^{-1}M)=\rho(M).
\]

If $J$ is a left ideal of $A$, use (\ref{f74}) with the same argument.
\end{proof}

\begin{lemma}
\label{3.9} Let $\ast\in\{f,c,b\}$, and let $A$ be a normed algebra. Suppose
that $J$ is $R_{\ast}(A)$ for $\ast\in\{f,c\}$ or an (one-sided)
$b$-quasinilpotent ideal of $A$ for $\ast=b$. If $M,N\in\mathcal{M}_{\ast}(A)$
and $M\subset N+J$ then $\rho(M)\leqslant\rho(N)$.
\end{lemma}

\begin{proof}
For any $a\in M$, denote by $a^{\prime}$ an element of $J$ such that
$a-a^{\prime}\in N$. Set $K=\{a^{\prime}:a\in M\}$ and, for any $\lambda
\in\mathbb{C}$, $N(\lambda)=\{a-\lambda a^{\prime}:a\in M\}$. Then the sets
$K$ and $N(\lambda)$ are in $\mathcal{M}_{\ast}(A)$, $K\subset J$ and
\[
N(1)\subset N.
\]
The function $f(\lambda)=\rho(N(\lambda))$ is subharmonic on $\mathbb{C}$ by
Theorem \ref{st2}. Since $N(\lambda)\subset2\operatorname{abs}(\lambda K\cup
M)$, we have, by Lemma \ref{st1} and one of Lemmas \ref{3.4}, \ref{3.20} and
\ref{3.21} corresponding to $\ast\in\{f,c,b\}$,
\[
\rho(N(\lambda))\leqslant2\rho(\operatorname{abs}(\lambda K\cup M))=2\rho
(\lambda K\cup M)=2\rho(M).
\]
Thus $f(\lambda)$ is constant and
\[
\rho(M)=f(0)=f(1)=\rho(N(1))\leqslant\rho(N).
\]
\end{proof}

Recall that if $J$ is a closed ideal of $A$ then, for any $M\subset A$, we
denote by $M/J$ the image of $M$ under the canonical epimorphism
$q_{J}:A\rightarrow A/J$.

\begin{theorem}
\label{3.10} Let $\ast\in\{f,c,b\}$, and let $A$ be a normed algebra. Suppose
that $J$ is $R_{\ast}(A)$ for $\ast\in\{f,c\}$ or a closed $b$-quasinilpotent
ideal of $A$ for $\ast=b$. Then $\rho(M)=\rho(M/J)$ for every $M\in
\mathcal{M}_{\ast}(A)$.
\end{theorem}

\begin{proof}
Let $M\in\mathcal{M}_{\ast}(A)$. For any $\varepsilon>0$, take
$n=n(\varepsilon)$ such that $\Vert M^{n}/J\Vert^{1/n}\leqslant\rho
(M/J)+\varepsilon$. It is a general geometric fact that, for arbitrary
$\delta>0$ , there are sets $N$ and $Q$ in $\mathcal{M}_{\ast}(A)$ with
$M^{n}\subset N+Q$, $Q\subset J$ and $\Vert N\Vert\leqslant\Vert M^{n}%
/J\Vert+\delta$ (note that it is evident for $\ast=f$ and $\ast=b$, the proof
for $\ast=c$ may be found in \cite[Lemma 6.9]{ST00}). Hence, by Lemma
\ref{3.9},
\[
\rho(M^{n})\leqslant\rho(N+Q)\leqslant\rho(N)\leqslant\Vert N\Vert
\leqslant\Vert M^{n}/J\Vert+\delta.
\]
Since $\delta$ is arbitrary,
\[
\rho(M)^{n}=\rho(M^{n})\leqslant\Vert M^{n}/J\Vert\leqslant(\rho
(M/J)+\varepsilon)^{n},
\]
whence $\rho(M)\leqslant\rho(M/J)+\varepsilon$. It follows that
\[
\rho(M)\leqslant\rho(M/J).
\]
The converse inequality is evident.
\end{proof}

It should be noted that the theorem does not hold for $J=$ $R_{b}$ and every
$M\in\mathcal{M}_{b}(A)$ (see (\ref{f66}) in the end of this subsection).
However, this theorem implies the equality
\[
\rho(M)=\rho(M/R_{b}(A)),
\]
for every precompact subset $M$ of $A$. Indeed, since $R_{b}(A)\subset
R_{c}(A)$, we have that
\[
\rho(M)=\rho(M/R_{c}(A))\leqslant\rho(M/R_{b}(A))\leqslant\rho(M).
\]

The next lemma shows in particular that elements of $R_{\ast}(A)$ have the
properties similar to ones of elements of the Jacobson radical.

\begin{lemma}
\label{3.11} Let $\ast\in\{f,c,b\}$, and let $A$ be a normed algebra. Suppose
that $J$ is $R_{\ast}(A)$ for $\ast\in\{f,c\}$ or a $b$-quasinilpotent ideal
of $A$ for $\ast=b$. Then $\rho(NM)=0$ and $\rho(N+M)=\rho(M)$ if one of the
following conditions holds.

\begin{itemize}
\item [(i)]$N\in\mathcal{M}_{\ast}(J)$ and $M\in\mathcal{M}_{\ast}(A)$.

\item[(ii)] $N$ is a precompact subset of $R_{b}(A)$ and $M$ is a bounded
subset of $A$.
\end{itemize}
\end{lemma}

\begin{proof}
Let (i) hold. Then $N+M\in\mathcal{M}_{\ast}(A)$. One may suppose that $J$ is
closed by Lemma \ref{3.0} in the case $\ast=b$. It follows from Theorem
\ref{3.10} that
\begin{equation}
\rho(N+M)=\rho((N+M)/J)=\rho(M/J)=\rho(M). \label{f70}%
\end{equation}
Furthermore, since $NM\subset(N\cup M)^{2}$,
\begin{equation}
\rho(NM)\leqslant\rho((N\cup M)^{2})=\rho(N\cup M)^{2}=\rho(M)^{2} \label{f75}%
\end{equation}
by one of Lemmas \ref{3.4}, \ref{3.20} and \ref{3.21} corresponding to
$\ast\in\{f,c,b\}$. Changing $N$ by $\lambda N$ we get that
\begin{equation}
\rho(NM)=0. \label{f76}%
\end{equation}

Now let (ii) hold. Note that (\ref{f75}) and (\ref{f76}) are valid since
$\rho(\lambda N\cup M)=\rho(M)$ by Lemma \ref{3.20}. Further, if
$G(\lambda)=M\cup({\lambda N}+M)$ then $G(\lambda)\subset2\operatorname{abs}%
(\lambda N\cup M)$, whence
\[
\rho(G(\lambda))\leqslant2\rho(\operatorname{abs}(\lambda N\cup M))=2\rho
(\lambda N\cup M)=2\rho(M)
\]
by Lemma \ref{3.20}. Since the function $\lambda\longmapsto\rho(G(\lambda))$
is bounded and subharmonic on $\mathbb{C}$, it is constant, whence
\[
\rho(N+M)\leqslant\rho(G(1))=\rho(G(0))=\rho(M).
\]
Changing $M$ by $-(N+M)$ in that inequality, we obtain that
\[
\rho(M)\leq\rho(N-(N+M))\leqslant\rho(-(N+M))=\rho(N+M),
\]
whence $\rho(N+M)=\rho(M)$.
\end{proof}

\begin{theorem}
\label{3.13}Let $A$ be a normed algebra. An element $a\in A$ belongs to
$R_{b}(A)$ (resp. $R_{c}(A)$) iff $\rho(aM)=0$, for any bounded (resp.
precompact) set $M\subset A$.
\end{theorem}

\begin{proof}
Let $\ast\in\{c,b\}$. If $a\in R_{\ast}(A)$ then $\rho(aM)=0$, for each
$M\in{\mathcal{M}}_{\ast}(A)$ by Lemma \ref{3.11}.

Proving the converse implication, suppose firstly that $||M||<1$. In this case
$\mathcal{S}_{1}(M)\in\mathcal{M}_{\ast}(A^{1})$ and, moreover, $\mathcal{S}%
_{1}(M)a\mathcal{S}_{1}(M)\in\mathcal{M}_{\ast}(A)$. Hence%
\[
\rho(a\mathcal{S}_{1}(M))^{2}=\rho((a\mathcal{S}_{1}(M))^{2})=\rho
(a\mathcal{S}_{1}(M)a\mathcal{S}_{1}(M))=0,
\]
whence $\rho(a\mathcal{S}_{1}(M))=0$. By Theorem \ref{3.12},
\[
\rho(\{a\}\cup M)=\rho(M).
\]
For $\Vert M\Vert\geqslant1$, we take $t>\Vert M\Vert$ and get
\[
\rho(\{a\}\cup M)=t\rho(\{a/t\}\cup(1/t)M)=t\rho((1/t)M)=\rho(M).
\]
\end{proof}

\begin{corollary}
\label{3.14}Let $\ast\in\{c,b\}$, and let $A$ be a normed algebra. Then
$R_{\ast}(A)$ contains all (one-sided) $\ast$-quasinilpotent ideals of $A$ and
$R_{c}(A)$ is the largest $c$-quasinilpotent ideal of $A$.
\end{corollary}

\begin{proof}
Note that $R_{\ast}(A)$ contains all (one-sided) $\ast$-quasinilpotent ideals
of $A$ by Lemma \ref{3.21}, and $R_{c}(A)$ is $c$-quasinilpotent itself by
Proposition \ref{3.8}, thus it is the largest $c$-quasinilpotent ideal of $A$.
\end{proof}

The following example shows that a Banach algebra need not have the largest
$b$-quasinilpotent ideal. Let $A_{n}$ denote the algebra of all strictly
upper-triangular matrices on $n$-dimensional Hilbert space $H_{n}$, for every
$n\in\mathbb{N}$, and let $A$ be the norm closure of the algebraic direct sum
$B$ of all algebras $A_{n}$ acting on the Hilbert space $H=\oplus H_{n}$.
Clearly $B$ is a union of nilpotent ideals, whence the largest $b$%
-quasinilpotent ideal would be equal to $A$; in particular $A=R_{b}(A)$ by
Corollary \ref{3.14}. On the other hand, if $M$ is the unit ball of $A$ then
$\Vert M^{n}\Vert=1$, for all $n$, hence $A$ is not $b$-quasinilpotent. In
particular
\begin{equation}
1=\rho(M)\neq\rho(M/R_{b}(A))=0. \label{f66}%
\end{equation}

\subsection{Radical-like properties of ideals $R_{f}$, $R_{c}$ and $R_{b}$}

\begin{lemma}
\label{3.15} Let $A$ be a normed algebra and $I$ its ideal. Then

\begin{itemize}
\item [(i)]$I\cap R_{f}(A)\subset R_{f}(I)$.

\item[(ii)] $I\cap R_{\ast}(A)=R_{\ast}(I)$, for $\ast\in\{c,b\}$.
\end{itemize}
\end{lemma}

\begin{proof}
The inclusions $I\cap R_{\ast}(A)\subset R_{\ast}(I)$ are evident.

Furthermore, let $a\in R_{\ast}(I)$ and $M\in{\mathcal{M}}_{\ast}(A)$, where
$\ast\in\{b,c\}$. Then $(aM)^{2}=aN$ where $N=MaM\in{\mathcal{M}}_{\ast}(I)$.
Hence $\rho((aM)^{2})=0$ implies $\rho(aM)=0$, whence $a\in R_{\ast}(A)$. We
proved that $R_{\ast}(I)\subset R_{\ast}(A)$. So $R_{\ast}(I)\subset R_{\ast
}(A)\cap I$. The converse is evident.
\end{proof}

\begin{lemma}
\label{3.16}Let $A$ be a normed algebra. Then

\begin{itemize}
\item [(i)]$R_{\ast}(R_{\ast}(A))=R_{\ast}(A)$, for $\ast\in\{f,c,b\}$.

\item[(ii)] $R_{\ast}({A}/{R_{\ast}(A)})=0$, for $\ast\in\{f,c\}$.
\end{itemize}
\end{lemma}

\begin{proof}
The statement (i) is evident.

For (ii), let $\widehat{a}=a/R_{\ast}(A)\in R_{\ast}(A/R_{\ast}(A))$. For
every $M\in\mathcal{M}_{\ast}(A)$, set $\widehat{M}=M/R_{\ast}(A)$ then , for
$\ast\in\{f,c\}$,
\begin{align*}
\rho(\{a\}\cup M)  &  =\rho((\{a\}\cup M)/R_{\ast}(A))=\rho(\{\widehat
{a}\}\cup\widehat{M})=\rho(\widehat{M})\\
&  =\rho(M)
\end{align*}
by Theorem \ref{3.10}. This shows that $a\in R_{\ast}(A)$ and accordingly
$\widehat{a}=0$.
\end{proof}

\begin{lemma}
\label{3.16.5} If $I$ is a closed ideal of a normed algebra $A$ then
$q_{I}(R_{\ast}(A))\subset R_{\ast}(A/I)$ for each $\ast\in\{f,b,c\}$.
\end{lemma}

\begin{proof}
Clearly, for any $M\in{\mathcal{M}}_{\ast}(A/I)$, there is $N\in{\mathcal{M}%
}_{\ast}(A)$ with $q_{I}(N)=M$. If $a\in R_{\ast}(A)$ then
\[
\rho(\{{\lambda}q_{I}(a)\}\cup M)=\rho(\{{\lambda}q_{I}(a)\}\cup
q_{I}(N))\leqslant\rho(\{{\lambda}a\}\cup N)=\rho(N)
\]
for every $\lambda\in\mathbb{C}$. By Lemma \ref{3.2}, $q_{I}(a)\in R_{\ast
}(A/I)$.
\end{proof}

Summing up the results of the previous lemmas we obtain the central statement
of this subsection.

\begin{theorem}
\label{3.17} $R_{f}$ satisfies properties $(1^{\circ})$, $(2^{\circ})$,
$(3^{\circ})$ and $(4^{\circ})$ (from the definition of a TR in Subsection
$\ref{sub0}$), $R_{c}$ is an HTR and $R_{b}$ is a hereditary UTR on $(NA)$.
\end{theorem}

\begin{proof}
The property (3$^{\circ}$) clearly holds for all $R_{\ast}$. The other
properties follow from Theorem \ref{3.7} and lemmas of this subsection.
\end{proof}

We will apply the term `radical' to all $R_{\ast}$ (taking in mind that only
$R_{c}$ is proved to be a TR on $(NA)$ indeed). The reason is that there are
universal classes of normed algebras on which all $R_{\ast}$ are HTRs; for
instance, they are the class of all finite-dimensional normed algebras, the
class of all commutative normed $Q$-algebras (see Proposition \ref{3.22}), etc.

\begin{proposition}
All $R_{\ast}$ are uniform and symmetric, $R_{f}$ is strong and not
semi-regular, $R_{c}$ and $R_{b}$ are regular and not strong.
\end{proposition}

\begin{proof}
Since a subalgebra of a $R_{\ast}$-radical algebra is $R_{\ast}$-radical,
uniformity of all $R_{\ast}$ is obvious. Also, it is clear that all $R_{\ast}$
are symmetric. The radicals $R_{c}$ and $R_{b}$ are regular (Lemma \ref{reg})
and in particular semi-regular. However $R_{f}$ is not semi-regular. Indeed,
there exists a normed algebra $C=\phi(A)$ for some continuous homomorphism
$\phi:A\rightarrow B$ from a topologically nilpotent Banach algebra $A$ to a
semisimple Banach algebra $B$, and $C$ is dense in $B$ \cite[Example 9.3]%
{D97}. Note that $R_{f}(C)=C$ and $R_{f}(\overline{C})=0$. On the other hand,
the same example shows that $R_{c}$ and $R_{b}$ are not strong radicals (in
virtue of $A=R_{c}(A)=R_{b}(A)$ and $R_{b}(C)\subset R_{c}(C)\subset
R_{c}(\overline{C})\subset\operatorname{Rad}(B)=0$), while $R_{f}$ is strong
by Lemma \ref{3.0}(ii).
\end{proof}

As a consequence of regularity, $R_{c}$ and $R_{b}$ are topologically
characteristic on $(NA)$ by Theorem \ref{p100}. Is $R_{f}$ topologically
characteristic on $(NA)$?

\begin{theorem}
\label{3.18}$R_{c}<R_{f}$ on $(NA)$ and $R_{b}<R_{c}$ on $(BA)$. Moreover,
$R_{b}$ is not a TR on $(BA)$.
\end{theorem}

\begin{proof}
Proposition \ref{3.8} shows that to see the distinction between $R_{f}$ and
$R_{c}$ it suffices to construct an $f$-quasinilpotent algebra which is not
$c$-quasinilpotent. Let $\phi:A\rightarrow B$ be the continuous homomorphism
of a radical Banach algebra $A$ on a dense subalgebra $C$ of a semisimple
Banach algebra $B$ constructed by P. G. Dixon \cite[Example 9.3]{D97}. Recall
that $A$ is topologically nilpotent, i.e. $b$-quasinilpotent, while $C$ is
$f$-quasinilpotent (by Lemma \ref{3.0}) and not $c$-quasinilpotent (since its
completion is not radical).

Let us prove that $R_{b}<R_{c}$ on $(BA)$. Let $G$ be the free non-unital
semigroup with countable set $\{x_{0},x_{1},...\}$ of generators. Let $J$ be
the ideal of $G$ consisting of all words which are not subwords of the
infinite word
\[
W=x_{0}x_{1}x_{0}x_{2}...x_{0}x_{n-1}x_{0}x_{n}\ldots
\]
($P$ is a \textit{subword} of $W$ if $W=RPH$, where $R\in G$ and $H$ is an
infinite word). Then $l^{1}(J)$ is isometrically imbedded as an ideal into
$l^{1}(G)$. Set $A=l^{1}(G)/l^{1}(J)$ and denote by $X_{i}$ the images of the
generators $x_{i}$ in $A$. Note that $X_{0}X_{0}=0$ and $X_{i}PX_{i}=0$ for
every monomial $P$ and $i>0$. Hence $X_{i}AX_{i}=0$ and
\[
(X_{i}M)^{2}=0
\]
for every bounded subset $M\subset A$ and $i>0$. It follows from Theorem
\ref{3.13} that
\[
X_{i}\in R_{b}(A)\subset R_{c}(A)
\]
for all $i>0$. Let $J$ be the closed ideal of $A$ generated by $N=\{X_{i}%
:i>0\}$. By Lemma \ref{3.15},
\[
J=R_{b}(J)=R_{c}(J).
\]
Further, $A/J$ is an one-dimensional algebra generated by a nilpotent. So
\[
A/J=R_{b}(A/J)=R_{c}(A/J).
\]
Since $R_{c}$ is a TR, $A=R_{c}(A)$ by Theorem \ref{2.3}(i). But $R_{b}(A)\neq
A$ since $X_{0}\notin R_{b}(A)$ in virtue of
\[
\left\|  (X_{0}N)^{n}\right\|  ^{1/n}=\left\|  X_{0}X_{1}\cdots X_{0}%
X_{n}\right\|  ^{1/n}=1
\]
for every $n\in\mathbb{N}$. This simultaneously proves that $R_{b}$ is not a
TR on $(BA)$ (otherwise $A=R_{b}(A)$ as in the case of $R_{c}$, a contradiction).
\end{proof}

The following questions are important. Do $R_{f}$ and $R_{c}$ coincide on
$(BA)$? Does $R_{f}$ coincide on $(BA)$ with $\operatorname{Rad}$?

\subsection{$\operatorname{Rad}^{t}$ as a TR related to joint quasinilpotence}

Since the tensor radical $\operatorname{Rad}^{t}$ is also related to a kind of
joint quasinilpotence it is natural to compare $R_{\ast}$ with
$\operatorname{Rad}^{t}$.

Let us firstly prove that $\operatorname{Rad}^{t}$ can be defined in the same
way as $R_{\ast}$.

\begin{theorem}
\label{des} Let $A$ be a Banach algebra. Then $a\in\operatorname{Rad}^{t}(A)$
iff $\rho_{1}(\{a\}\cup M)=\rho_{1}(M)$ for every summable subset $M$ of $A$.
\end{theorem}

\begin{proof}
If $a\in\operatorname{Rad}^{t}(A)$ then, by Theorem \ref{3quot},
\begin{align*}
\rho_{1}(\{a\}\cup M)  &  =\rho_{1}((\{a\}\cup M)/\operatorname{Rad}%
^{t}(A))=\rho_{1}(M/\operatorname{Rad}^{t}(A))\\
&  =\rho_{1}(M).
\end{align*}

Conversely, if $\rho_{1}(\{a\}\cup M)=\rho_{1}(M)$ then
\[
\rho_{1}((\{a\}\cup M)^{2})=\rho_{1}(\{a\}\cup M)^{2}=\rho_{1}(M)^{2}%
\]
by (\ref{f62}). Since $aM\subset(\{a\}\cup M)^{2}$ as generalized subsets, we
have
\[
\rho_{1}(aM)\leqslant\rho_{1}(M)^{2}.
\]
Changing $a$ by ${\lambda}a$ (which clearly satisfies the same condition) for
$\lambda\in\mathbb{C}$, we get:
\[
|\lambda|\rho_{1}(aM)\leqslant\rho_{1}(M)^{2},
\]
whence $\rho_{1}(aM)=0$ and it remains to apply Theorem \ref{1quas}.
\end{proof}

Now we will prove that the radical $R_{c}$ is weakly tensor.

\begin{theorem}
\label{p40}If $A$ is a compactly quasinilpotent Banach algebra then
$A\widehat{\otimes}B$ is compactly quasinilpotent, for any Banach algebra $B$.
\end{theorem}

\begin{proof}
If $M$ is a precompact set in $A{\widehat{\otimes}}B$ then by \cite[Corollary
7.2.2]{RR64} there exists a sequence $a_{n}\otimes b_{n}$ with $a_{n}%
\rightarrow0,b_{n}\rightarrow0$ such that $M\subset\operatorname{abs}%
\{a_{n}\otimes b_{n}:n\geqslant1\}$. Hence
\begin{align*}
\rho(M)  &  \leqslant\rho(\operatorname{abs}\{a_{n}\otimes b_{n}%
:n\geqslant1\})=\rho(\{a_{n}\otimes b_{n}:n\geqslant1\})\\
&  \leqslant\rho(P\otimes Q),
\end{align*}
where $P=\{a_{n}:n\geqslant1\}$ and $Q=\{b_{n}:n\geqslant1\}$. But
\[
\Vert(P\otimes Q)^{n}\Vert=\Vert P^{n}\otimes Q^{n}\Vert\leqslant\Vert
P^{n}\Vert\Vert Q^{n}\Vert
\]
whence $\rho(P\otimes Q)\leqslant\rho(P)\rho(Q)$. In our assumptions
$\rho(P)=0$ because $P$ is precompact. Hence $\rho(M)=0$.
\end{proof}

It is clear that any summable set $M$ of a normed algebra is precompact and
\[
\rho(M)\leqslant\rho_{1}(M)\leqslant\Vert M\Vert_{1}.
\]
Moreover, if $M$ is finite then
\[
\rho_{1}(M)\leqslant k\rho(M),
\]
where $k=\operatorname{card}(M)$. It follows that if $M$ is summable with
$\rho_{1}(M)=0$ then $\rho(M)=0$, and if $M$ is finite with $\rho(M)=0$ then
$\rho_{1}(M)=0$.

\begin{corollary}
\label{p41}$R_{c}\leqslant\operatorname{Rad}^{t}\leqslant R_{f}$ on $(BA)$.
\end{corollary}

\begin{proof}
If $a\in\operatorname{Rad}^{t}(A)$ then $\rho_{1}(\{\lambda a\}\cup
M)=\rho_{1}(M)$ for every finite subset $M$ of $A$ and $\lambda\in\mathbb{C}$
by Theorem \ref{des}, whence
\[
\sup\{\rho(\{\lambda a\}\cup M):\lambda\in\mathbb{C\}\leqslant}\sup\{\rho
_{1}(\{\lambda a\}\cup M):\lambda\in\mathbb{C\}=}\rho_{1}(M).
\]
By Lemma \ref{3.2}, $a\in R_{f}(A)$.

It follows from Theorem \ref{f40} that if $A=R_{c}(A)$ then
$A=\operatorname{Rad}^{t}(A)$. Then
\[
R_{c}(A)=\operatorname{Rad}^{t}(R_{c}(A))\subset\operatorname{Rad}^{t}(A)
\]
because $R_{c}(A)$ is a closed ideal of $A$ and $\operatorname{Rad}^{t}$ is a
TR on $(BA)$.
\end{proof}

So, it can be said that the problem of the tensor radicality of a radical
Banach algebra is ``intermediate'' between the problems of finite
quasinilpotence and compact quasinilpotence. It follows from Corollary
\ref{p41} and Lemma \ref{reg} that
\[
\rho_{1}(M)=0
\]
for every summable subset $M$ of $R_{c}(A)$, for a normed algebra $A$.

\subsection{$p$-quasinilpotent algebras}

We will conclude this section by consideration of a natural scale of joint
spectral radii, including $\rho$ and $\rho_{1}$ as polar extrema. A close (but
not identical) notion see in \cite{P97}.

Let $A$ be a normed algebra and $p\geqslant1$. For a generalized subset $M$ of
$A$ with a representative $(a_{\alpha})_{\alpha\in\Lambda}$, let
\[
\Vert M\Vert_{p}=(\sum_{\alpha\in\Lambda}\Vert a_{\alpha}\Vert^{p})^{1/p}.
\]
As in the case $p=1$, $\Vert M\Vert_{p}$ does not depend on a representative.
Moreover,
\begin{equation}
\Vert MN\Vert_{p}\leqslant\Vert M\Vert_{p}\Vert N\Vert_{p} \label{f90}%
\end{equation}
for generalized subsets of $A$ (recall that product $MN$ and a power $M^{n}$
are defined here by rules of generalized subsets; see Subsection \ref{sub}). A
generalized subset $M$ of $A$ is called $p$\textit{-summable} if $\Vert
M\Vert_{p}<\infty$. For a $p$-summable generalized subset $M$ of $A$, it
follows from (\ref{f90}) that there exists
\[
\rho_{p}(M)=\lim(\Vert M^{n}\Vert_{p})^{1/n}=\inf(\Vert M^{n}\Vert_{p}%
)^{1/n}.
\]
Changing as usually sums to suprema for $p=\infty$, one can write that
\[
\rho(M)=\rho_{\infty}(M).
\]
Clearly any $p_{1}$-summable generalized subset $M$ is $p_{2}$-summable if
$p_{1}<p_{2}$; moreover
\[
\Vert M\Vert_{p_{1}}\geqslant\Vert M\Vert_{p_{2}}.
\]
It follows that
\begin{equation}
\rho_{p_{1}}(M)\geqslant\rho_{p_{2}}(M). \label{f10}%
\end{equation}

A normed algebra $A$ is called $p$-\textit{quasinilpotent} if $\rho_{p}(M)=0$,
for every $p$-summable subset $M$ of $A$. It is clear that if $A$ is
$p$-quasinilpotent then $\rho_{p}(M)=0$ for every $p$-summable generalized
subset of $A$.

\begin{theorem}
\label{4.17} Let $A$ be a Banach algebra and $p>1$. Then

\begin{itemize}
\item [(i)]If $A$ is $p$-quasinilpotent then $A$ is $1$-quasinilpotent.

\item[(ii)] If $A$ is $1$-quasinilpotent then $\rho_{p}(M)=0$, for any
$1$-summable subset $M$ of $A$.
\end{itemize}
\end{theorem}

\begin{proof}
(i) Note that the case $p=\infty$ follows from Corollary \ref{p41}. Let
$p<\infty$. It suffices to prove that $A$ is $\operatorname{Rad}^{t}$-radical.
Let $B$ be an arbitrary Banach algebra and $T=\sum a_{k}\otimes b_{k}\in
A{\widehat{\otimes}}B$ with $\sum\left\|  a_{k}\right\|  \left\|
b_{k}\right\|  <\infty$. Multiplying by suitable scalars, one may suppose that
$\sum\left\|  a_{k}\right\|  ^{p}<\infty$ and $\sum\left\|  b_{k}\right\|
^{q}<\infty$, where $q=p/(p-1)$. It follows that
\begin{align*}
\Vert T^{n}\Vert &  =\left\|  \sum a_{k_{1}}...a_{k_{n}}{\otimes}b_{k_{1}%
}...b_{k_{n}}\right\| \\
&  \leqslant(\sum\Vert a_{k_{1}}...a_{k_{n}}\Vert^{p})^{1/p}(\sum\Vert
b_{k_{1}}...b_{k_{n}}\Vert^{q})^{1/q}\leqslant\Vert M^{n}\Vert_{p}\beta^{n},
\end{align*}
where $M$ is a generalized subset with the representative $(a_{k}%
)_{k\geqslant1}$ and
\[
\beta=(\sum\Vert b_{k}\Vert^{q})^{1/q}.
\]
Thus
\[
\Vert T^{n}\Vert^{1/n}\leqslant\beta(\Vert M^{n}\Vert_{p})^{1/n},
\]
whence $\rho(T)\leqslant\rho_{p}(M)=0$.

(ii) follows immediately from (\ref{f10}).
\end{proof}


\begin{thebibliography}{99}
\bibitem{A79}B. Aupetit, Propri\'{e}t\'{e}s spectrales des alg\`{e}bres de
Banach, Lecture Notes in Mathematics 735, Springer-Verlag,
Berlin-Heidelberg-New York, 1979.

\bibitem{BD73}F. F. Bonsall and J. Duncan, Complete normed algebras,
Springer-Verlag, Berlin, 1973.

\bibitem{B67}N. Bourbaki, \'{E}l\'{e}ments de math\'{e}matique, Th\'{e}ories
Spectrales, Hermann. 1967.

\bibitem{D00}H. G. Dales, Banach algebras and automatic continuity, Clarendon
Press, Oxford, 2000.

\bibitem{D65}N. J. Divinsky, Rings and radicals, Allen and Unwin, London, 1965.

\bibitem{D91}P. G. Dixon, Topologically nilpotent Banach algebras and
factorization, \textit{Proc. Royal Soc. Edinburgh} A \textbf{119} (1991), 329-341.

\bibitem{D97}P. G. Dixon, Topologically irreducible representations and
radicals in Banach algebras, \textit{Proc. London Math Soc.} (3) \textbf{74
}(1997) 174-200.

\bibitem{DM92}P. G. Dixon and V. M\"{u}ller, A note on topologically nilpotent
Banach algebras, \textit{Studia Math.} \textbf{102} (1992), 269-275.

\bibitem{DW92}P. G. Dixon and G. A. Willis, Approximate identities in
extensions of topologically nilpotent Banach algebras, \textit{Proc. Royal
Soc. Edinburgh} A \textbf{122} (1992), 45-52.

\bibitem{HK76}W. K. Hayman and P. B. Kennedy, Subharmonic functions, Volume 1,
LMS Monographs 9, Academic press, London-New York-San Francisko, 1976

\bibitem{K48}I. Kaplansky, Locally compact rings, \textit{Amer. Math. J.
}\textbf{70 }(1948), 447-459.

\bibitem{KS97}E. Kissin and V. S. Shulman, Representations on Krein spaces and
derivations of $C^{\ast}$-algebras, Pitman monographs and Surveys in Pure and
Applied Math. 89, Addison Wesley Longman, London-New York, 1997.

\bibitem{P94}T. W. Palmer, Banach algebras and the general theory of
*-algebras, V. 1, Cambridge Univ. Press, 1994.

\bibitem{PW99}J. R. Peters and R. W. Wogen, Commutative radical operator
algebras, \textit{J.Operator Theory} \textbf{42 }(1999), 405-424.

\bibitem{P97}V. Yu. Protasov, The generalized joint spectral radius. A
geometric approach, \textit{Izv. Math.} \textbf{61 }(5) (1997), 995-1030.

\bibitem{R97}C. J. Read, Quasinilpotent operators and the invariant subspace
problem, \textit{J. London Math. Soc.} (2) \textbf{56} (1997), 595-606.

\bibitem{RR64}A. P. Robertson and W. Robertson, Topological vector spaces,
Cambridge Univ. Press, 1964.

\bibitem{RS60}G.-C. Rota and W. G. Strang, A note on the joint spectral
radius, \textit{Indag. Math.} \textbf{22} (1960), 379-381.

\bibitem{ST00}V. S. Shulman and Yu. V. Turovskii, Joint spectral radius,
operator semigroups and a problem of a Wojtynski, \textit{J. Funct. Anal.}
\textbf{177} (2000) 383-441.

\bibitem{ST02}V. S. Shulman and Yu. V. Turovskii, Formulae for joint spectral
radii of sets of operators, \textit{Studia Math.} \textbf{149} (2002), 23-37.

\bibitem{S69}A. M. Sinclair, Continuous derivations on Banach algebras,
\textit{Proc. Amer. Math. Soc.} \textbf{20 }(1969) 166-170.

\bibitem{S81}F. A. Sz\'{a}sz, Radicals of rings, Akad\'{e}miai Kiad\'{o},
Budapest, 1981.

\bibitem{T99}Yu. V. Turovskii, Volterra semigroups have invariant subspaces,
\textit{J. Funct. Anal. }\textbf{162 }(1999), 313-323.
\end{thebibliography}
\end{document}